\documentclass[11pt]{amsart}
\usepackage{amscd,amsfonts,amsmath,amssymb,amsthm,enumerate,epsfig,flexisym,float,graphicx,hhline,hyperref,pdfpages,subcaption,tikz,url}
\usepackage[autostyle]{csquotes}
\usepackage[all,cmtip]{xy}

\makeatletter
\@namedef{subjclassname@2020}{\textup{2020} Mathematics Subject Classification}
\makeatother

\newtheorem{theorem}{Theorem}[section]
\newtheorem{corollary}[theorem]{Corollary}
\newtheorem{lemma}[theorem]{Lemma}
\newtheorem{proposition}[theorem]{Proposition}
\theoremstyle{definition}

\newtheorem{definition}[theorem]{Definition}
\newtheorem{example}[theorem]{Example}

\newtheorem{remark}[theorem]{Remark}
\numberwithin{equation}{subsection}
\newtheorem*{ack}{Acknowledgement}

\newcommand{\Conj}{\operatorname{Conj}}
\newcommand{\Core}{\operatorname{Core}}
\newcommand{\Alex}{\operatorname{Alex}}

\newcommand{\Env}{\operatorname{Env}}
\newcommand{\Inn}{\operatorname{Inn}}

\newcommand{\Map}{\operatorname{Map}}

\newcommand{\Z}{\operatorname{Z}}
\newcommand{\C}{\operatorname{C}}
\newcommand{\Aut}{\operatorname{Aut}}
\newcommand{\R}{\operatorname{R}}

\newcommand{\id}{\operatorname{id}}

\DeclareFontFamily{U}{mathb}{\hyphenchar\font45}
\DeclareFontShape{U}{mathb}{m}{n}{<5> <6> <7> <8> <9> <10> <10.95> <12> <14.4> <17.28> <20.74> <24.88> mathb10}{}
\DeclareSymbolFont{mathb}{U}{mathb}{m}{n}
\DeclareMathSymbol{\blackdiamond}{2}{mathb}{"0C}

\usetikzlibrary{cd, arrows, knots, patterns, positioning, shapes.callouts, shapes.geometric, shapes.symbols}
\tikzset{>=stealth', rect1/.style={rectangle, draw=black, text width=38mm, minimum height=17mm, text centered},
rect2/.style={rectangle, draw=black, text width=55mm, minimum height=28mm, text centered},
arrow/.style={->}}

\begin{document}
	
\title[Orderability of link quandles]{Orderability of link quandles}
\author[H. Raundal]{Hitesh Raundal}
\author[M. Singh]{Mahender Singh}
\author[M. Singh]{Manpreet Singh}
\address{Department of Mathematical Sciences, Indian Institute of Science Education and Research (IISER) Mohali, Sector 81, S. A. S. Nagar, P. O. Manauli, Punjab 140306, India.}
\email{hiteshrndl@gmail.com}
\email{mahender@iisermohali.ac.in}
\email{mp15009@iisermohali.ac.in}


\begin{abstract}
The paper develops a general theory of orderability of quandles with a focus on link quandles of tame links and gives some general constructions of orderable quandles. We prove that knot quandles of many fibered prime knots are right-orderable, whereas link quandles of most non-trivial torus links are not right-orderable. As a consequence, we deduce that the knot quandle of the trefoil is neither left nor right orderable. Further, it is proved that link quandles of certain non-trivial positive (or negative) links are not bi-orderable, which includes some alternating knots of prime determinant and alternating Montesinos links. The paper also explores interconnections between orderability of quandles and that of their enveloping groups. The results establish that orderability of link quandles behave quite differently than that of corresponding link groups.
\end{abstract}

\maketitle

\section{Introduction}\label{introduction}

It is known that existence of a linear order on a group has profound implications on its structure. For instance, a left-orderable group cannot have torsion and a bi-orderable group cannot have even generalized torsion (product of conjugates of a non-trivial element being trivial). In terms of applicability, it is known that integral group rings of left-orderable groups have no zero-divisors. Recall that the famous Kaplansky's conjecture asserts this to be true for all torsion-free groups. Concerning groups arising in topology, the literature shows that many such groups are left-orderable. In fact, the fundamental group of any connected surface other than the projective plane or the Klein bottle is bi-orderable \cite{brw}. Braid groups are the most relevant examples of left-orderable groups which are not bi-orderable \cite{d}. On the other hand, pure braid groups are known to be bi-orderable \cite{fr}. Rourke and Wiest \cite{rw} extended this result by showing that mapping class groups of all Riemann surfaces with non-empty boundary are left-orderable. In general these groups are not bi-orderable. Orderability of 3-manifold groups has been investigated extensively where left-orderability is a rather common property. Concerning link complements, it is known that fundamental groups of link complements are left-orderable \cite{brw}, whereas fundamental groups of not all link complements are bi-orderable \cite{pr}. For example, the knot group of the figure-eight knot is bi-orderable and the group of a non-trivial cable of an arbitrary knot is not bi-orderable. In general, a fibered knot has bi-orderable knot group if all the roots of its Alexander polynomial are real and positive \cite{pr}. Note that there are infinitely many such fibered knots. For more on the literature, we refer the reader to the recent monograph \cite{cr} by Clay and Rolfsen which explores orderability of groups motivated by topology, like fundamental groups of surfaces or 3-manifolds, braid and mapping class groups, groups of homeomorphisms, etc. Another monograph \cite{ddrw} on orderability of braid groups is worth looking into.
\par

The notion of orderability can be defined for magmas just as for groups. Since quandles represent interesting examples of non-associative magmas and link quandles are deeply related to link groups, it seems natural to explore orderability of quandles. Orderability of magmas including conjugation quandles of bi-orderable groups has been analyzed in \cite{DDHPV}. Further, in recent works \cite{Ha2018, HaValentina}, the space of right orders on the conjugation quandle of the countably infinite rank free group has been shown to be the Cantor set. Another recent work \cite{bps} investigated orderability of quandles for studying zero-divisors in quandle rings. Unlike groups, one sided orderability of quandles does not imply the other sided orderability. We also explore interconnections between orderability of quandles and that of their enveloping groups. Our results show that orderability of link quandles behave quite differently than that of corresponding link groups. For instance, we show that the knot quandle of the trefoil knot is neither left nor right orderable, whereas knot quandles of many fibered prime knots, for example the figure-eight knot, are right-orderable.
\medskip

The paper is organized as follows. Section \ref{definitions-known-results} recalls some basic definitions, examples and results that are needed in subsequent sections. In Section \ref{properties-orderable-quandle}, we derive some basic properties of orderings on quandles and prove that any linear order on a quandle must be of restricted type (Theorem \ref{th2}). In Section \ref{construction-order-quandles-automorphisms}, we define action of a quandle and prove that if a quandle acts faithfully on a well-ordered set by order-preserving bijections, then it is right-orderable (Theorem \ref{th5}). The converse holds for right-orderable semi-latin quandles. We show that certain disjoint unions, direct products and extensions of orderable quandles are orderable (Proposition \ref{prop6}, Proposition \ref{prop1} and Proposition \ref{orderable-extensions}). We also obtain the group of order-preserving automorphisms of natural orderable-quandles arising from orderable groups (Proposition \ref{order-aut-conj} and Proposition \ref{order-aut-alex}). Section \ref{orderability-general-quandle} discusses orderability of some general quandles. We determine a necessary condition for the left-orderability of a quandle $Q$ for which the natural map $Q\to\Conj(\Env(Q))$ is injective (Proposition \ref{prop11}). The construction of free racks/quandles has been extended in a recent work of Bardakov and Nasybullov \cite{bn} to so called $(G,A)$-racks/quandles. We prove that if $G$ is a bi-orderable group and $A$ a subset of $G$, then the corresponding $(G,A)$-racks/quandles are right-orderable (Theorem \ref{th3}). As a consequence, it follows that free quandles, in particular, quandles of trivial links are right-orderable (Corollary \ref{cor2}). In Section \ref{orderability-knot-quandle}, we prove that if all the roots of the Alexander polynomial of a fibered prime knot are real and positive, then its knot quandle is right-orderable (Corollary \ref{cor4}). In another main result of this section, we prove that link quandles of certain non-trivial positive (or negative) links are not bi-orderable (Theorem \ref{th6}). This includes non-trivial prime Montesinos links that are alternating positive (or negative) and knots of prime determinant that are alternating positive (or negative). In Section \ref{orderability-torus-quandle}, we prove that if $m,n\geq2$ are integers such that one is not a multiple of the other, then the link quandle of the torus link $T(m,n)$ is not right-orderable (Theorem \ref{th4}). As a consequence, we recover a result of Perron and Rolfsen that the knot group of a non-trivial torus knot is not bi-orderable (Corollary \ref{cor3}). Finally, in Section \ref{involutory-quandles-links}, we discuss left-orderability of involutory quandles of alternating links (Theorem \ref{thm-left-order}).

\medskip

\section{Preliminaries}\label{definitions-known-results}

We begin by defining the main object of our study.

\begin{definition}
A {\it quandle} is a non-empty set $Q$ together with a binary operation $*$ satisfying the following axioms:
\begin{enumerate}
\item[{\textbf Q1}] $x*x=x$\, for all $x\in Q$.
\item[{\textbf Q2}] For each $x, y \in Q$, there exists a unique $z \in Q$ such that $x=z*y$.
\item[{\textbf Q3}] $(x*y)*z=(x*z)*(y*z)$\, for all $x,y,z\in Q$.
\end{enumerate}
\end{definition}

The axiom {\textbf Q2} is equivalent to bijectivity of the right multiplication by each element of $Q$. This gives a dual binary operation $*^{-1}$ on $Q$ defined as $x*^{-1}y=z$ if $x=z*y$. Thus, the axiom {\textbf Q2} is equivalent to saying that
\begin{equation*}
(x*y)*^{-1}y=x\qquad\textrm{and}\qquad\left(x*^{-1}y\right)*y=x
\end{equation*}
for all $x,y\in Q$, and hence it allows us cancellation from right. The axioms {\textbf Q1} and {\textbf Q3} are referred as idempotency and distributivity axioms, respectively. Idempotency and cancellation together give $x*^{-1}x=x$ for all $x\in Q$.
\par

Topologically, the three quandle axioms correspond to the three Reidemeister moves of planar diagrams of links in the 3-space, which was observed independently in the foundational works of Joyce \cite{j1,j2} and Matveev \cite{m}. Following are examples of quandles, some of which we shall use in the forthcoming sections.
\begin{itemize}
\item If $G$ is a group and $n \in \mathbb{Z}$, then the binary operation $x*y=y^{-n}xy^n$ turns $G$ into the quandle $\Conj_n(G)$ called the $n$-\textit{conjugation quandle} of $G$. For $n=1$, the quandle is simply denoted by $\Conj(G)$.
\item If $G$ is a group, then the binary operation $x*y=yx^{-1}y$ turns $G$ into the quandle $\Core(G)$ called the \textit{core quandle} of $G$. In particular, if $G$ is a cyclic group of order $n$, then it is called the \textit{dihedral quandle} and is denoted by $\R_n$. Usually, one writes $\R_n=\{0, 1, \ldots, n-1\}$ with $i*j=2j-i \mod n$.
\item If $G$ is a group and $\phi\in \Aut(G)$, then $G$ with the binary operation $x*y=\phi\left(xy^{-1}\right)y$ forms a quandle $\Alex(G,\phi)$ referred as the \textit{generalized Alexander quandle} of $G$ with respect to $\phi$.
\item If $L$ is a link in the 3-sphere, then Joyce \cite{j1,j2} and Matveev \cite{m} associated a quandle $Q(L)$ to $L$ called the \textit{link quandle} of $L$. We fix a diagram $D(L)$ of $L$ and label its arcs. Then the link quandle $Q(L)$ is generated by labelings of arcs of $D(L)$ with a defining relation at each crossing in $D(L)$ given as shown in Figure \ref{fig1}. The link quandle of a link $L$ is independent of the diagram chosen, i.e., the quandles obtained from any two diagrams of $L$ are isomorphic.
\begin{figure}[H]
\begin{subfigure}{0.4\textwidth}
\centering
\begin{tikzpicture}[scale=0.6]
\node at (0.4,-1.2) {{\small $x$}};
\node at (-1.2,0.4) {{\small $y$}};
\node at (0.8,1.2) {{\small $x*y$}};
\begin{knot}[clip width=6, clip radius=4pt]
\strand[->] (-2,0)--(2,0);
\strand[->] (0,-2)--(0,2);
\end{knot}
\end{tikzpicture}
\caption{At a positive crossing}
\end{subfigure}
\begin{subfigure}{0.4\textwidth}
\centering
\begin{tikzpicture}[scale=0.6]
\node at (1.1,1.2) {{\small $x*^{-1}y$}};
\node at (-1.2,0.4) {{\small $y$}};
\node at (0.4,-1.2) {{\small $x$}};
\begin{knot}[clip width=6, clip radius=4pt]
\strand[->] (2,0)--(-2,0);
\strand[->] (0,-2)--(0,2);
\end{knot}
\end{tikzpicture}
\caption{At a negative crossing}
\end{subfigure}
\caption{Relations at a positive and at a negative crossing}
\label{fig1}
\end{figure}
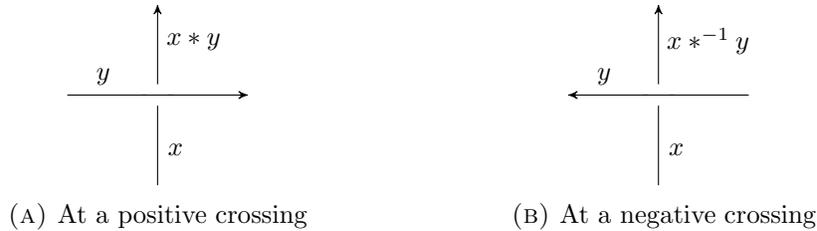
\end{itemize}

A \textit{homomorphism of quandles} $P$ and $Q$ is a map $\phi:P\to Q$ with $\phi(x*y)=\phi(x)*\phi(y)$ for all $x,y\in P$. By cancellation in $P$ and $Q$, we obtain $\phi(x*^{-1}y)=\phi(x)*^{-1}\phi(y)$ for all $x,y\in P$. We denote the group of all automorphisms of a quandle $Q$ by $\Aut(Q)$. The quandle axioms are equivalent to saying that for each $y\in Q$, the map $S_y:Q\to Q$ (called the symmetry at $y$) given by $S_y(x)=x* y$ is an automorphism of $Q$ fixing $x$. The group $\Inn(Q)$ generated by such automorphisms is called the group of {\it inner automorphisms} of $Q$. 
\medskip

We recall some relevant properties of quandles. A quandle $Q$ is said to be 
\begin{itemize}
\item {\it connected} if the group $\Inn(Q)$ acts transitively on $Q$. For example, the dihedral quandle $\R_{2n+1}$ is connected, whereas $\R_{2n}$ is not. Henceforth, the term {\it orbit} would correspond to an orbit in $Q$ under the action of $\Inn(Q)$. 
\item {\it involutory} if $x*^{-1}y=x* y$ for all $x,y\in Q$. For example, the core quandle $\Core(G)$ of any group $G$ is involutory. One can define the involutory quotient of a given quandle $Q$ as the quandle obtained from $Q$ by imposing relations of the form $x*^{-1}y=x* y$ for all $x, y \in Q$.
\item {\it commutative} if $x*y=y*x$ for all $x,y\in Q$.
\item {\it quasi-commutative} if for given $x,y\in Q$, at least one of the following hold: $x*y=y*x$, $x*y=y*^{-1}x$, $x*^{-1}y=y*x$ or $x*^{-1}y=y*^{-1}x$. Obviously, every commutative quandle is quasi-commutative. The Alexander quandle $\Alex(\mathbb{R},\phi_2)$ of Example \ref{ex1} is quasi-commutative but not commutative.
\item {\it latin} if the left multiplication $L_x:Q\to Q$ defined by $L_x(y)=x* y$ is a bijection for each $x\in Q$. 
\item {\it semi-latin} if $L_x$ is an injective map for each $x\in Q$. Each latin quandle is obviously semi-latin, but the converse is not true in general. For example, the quandle $\Core(\mathbb{Z})$ is semi-latin but not latin. 
\item {\it simple} if for any quandle $P$, every homomorphism $Q \to P$ is either injective or constant. For example, the dihedral quandle $\R_3$ is commutative, latin and simple. On the other hand, the dihedral quandle $\R_{2n}$ is not commutative or latin or simple. 
\end{itemize}
\medskip

The \textit{enveloping group} $\Env(Q)$ of a quandle $Q$ is the group with the set of generators as $\{\tilde{x}~|~x \in Q\}$ and the defining relations as
\begin{equation*}
\widetilde{x*y}=\tilde{y}^{-1} \tilde{x} \tilde{y}
\end{equation*}
for all $x,y\in Q$. For example, if $Q$ is a trivial quandle, then $\Env(Q)$ is the free abelian group of rank equal to the cardinality of $Q$. By \cite{j1,j2}, the enveloping group of the link quandle $Q(L)$ of a link $L$ is the link group $G(L)$ of $L$. The natural map
\begin{equation*}
\eta: Q \to \Env(Q)
\end{equation*}
given by $\eta(x)=\tilde{x}$ is a quandle homomorphism from $Q$ to $\Conj (\Env(Q))$. The map $\eta$ is not injective in general. The presentation of the enveloping group of a quandle can be reduced as follows \cite[Theorem 5.1.7]{w}.

\begin{theorem}\label{th1}
Let $Q$ be a quandle with a presentation $Q= \langle X~~|~~R \rangle$. Then its enveloping group has a presentation $\Env(Q)\cong\langle \tilde{x}, ~~x \in X~~|~~\tilde{R} \rangle$, where $\tilde{R}$ consists of relations obtained from relations in $R$ with an expression $x*y$ replaced by ${\tilde{y}}^{\,-1}\tilde{x}\tilde{y}$ and an expression $x*^{-1}y$ replaced by $\tilde{y}\tilde{x}{\tilde{y}}^{\,-1}$.
\end{theorem}
\par

If $Q$ is a quandle, then by \cite[Lemma 4.4.7]{w}, we can write
\begin{equation*}
x*^d\left(y*^e z\right)=\left(\left(x*^{-e} z\right)*^dy\right)*^ez\quad\textrm{(called the {\it left association identity})}
\end{equation*}
for all $x,y,z\in Q$ and $d,e\in\{-1,1\}$. Henceforth, we will write a left-associated product
\begin{equation*}
\left(\left(\cdots\left(\left(a_0*^{e_1}a_1\right)*^{e_2}a_2\right)*^{e_3}\cdots\right)*^{e_{n-1}}a_{n-1}\right)*^{e_n}a_n
\end{equation*}
simply as 
\begin{equation*}
a_0*^{e_1}a_1*^{e_2}\cdots*^{e_n}a_n.
\end{equation*}

A repeated use of left association identity gives the following result \cite[Lemma 4.4.8]{w}.

\begin{lemma}\label{lem4}
The product 
\begin{equation*}
\left(a_0*^{d_1}a_1*^{d_2}\cdots*^{d_m}a_m\right)*^{e_0}\left(b_0*^{e_1}b_1*^{e_2}\cdots*^{e_n}b_n\right)
\end{equation*}
\noindent of two left-associated forms $a_0*^{d_1}a_1*^{d_2}\cdots*^{d_m}a_m$ and $b_0*^{e_1}b_1*^{e_2}\cdots*^{e_n}b_n$ in a quandle can again be written in a left-associated form as
\begin{equation*}
a_0*^{d_1}a_1*^{d_2}\cdots*^{d_m}a_m*^{-e_n}b_n*^{-e_{n-1}}b_{n-1}*^{-e_{n-2}}\cdots*^{-e_1}b_1*^{e_0}b_0*^{e_1}b_1*^{e_2}\cdots*^{e_n}b_n.
\end{equation*}
\end{lemma}

Thus, any product of elements of a quandle $Q$ can be expressed in the canonical left-associated form $a_0*^{e_1}a_1*^{e_2}\cdots*^{e_n}a_n$, where $a_0\neq a_1$, and for $i=1,2,\ldots,n-1$, $e_i=e_{i+1}$ whenever $a_i=a_{i+1}$.
\medskip

We now define orderability of quandles.

\begin{definition} 
A quandle $Q$ is said to be \textit{left-orderable} if there is a (strict) linear order $<$ on $Q$ such that $x<y$ implies $z*x<z*y$ for all $x,y,z\in Q$. Similarly, a quandle $Q$ is \textit{right-orderable} if there is a linear order $<^\prime$ on $Q$ such that $x<^\prime y$ implies $x*z<^\prime y*z$ for all $x,y,z\in Q$. A quandle is \textit{bi-orderable} if it has a linear order with respect to which it is both left and right ordered.
\end{definition}

For example, a trivial quandle with more than one element is right-orderable but not left-orderable. If $Q=\left\{x_1,x_2,\ldots\right\}$ is a trivial quandle, then it is clear that the linear order $x_1<x_2<\cdots$ is preserved under multiplication on the right, but is not preserved under multiplication on the left. Notice the contrast to groups where left-orderability implies right-orderability and vice-versa.

\begin{proposition}\label{prop9}
Any non-trivial left or right orderable quandle is infinite.
\end{proposition}

\begin{proof}
Let $Q$ be a non-trivial quandle that is right-orderable. Then there exist elements $x \neq y$ in $Q$ such that $S_y(x)\neq x$. It follows from \cite[Proposition 3.7]{bps} that the $\left\langle S_y\right\rangle$-orbit of $x$ is infinite, and thus $Q$ must be infinite. On the other hand, if $Q$ is left-orderable, then by \cite[Proposition 3.7]{bps}, the set $\left\{L_y^n(x)\mid n=1,2,\ldots\right\}$ is infinite for any $x\neq y$ in $Q$, and hence $Q$ must be infinite.
\end{proof}

It also follows from \cite[Proposition 3.7]{bps} that a non-trivial involutory quandle is not right-orderable. A large number of left or right-orderable quandles can be constructed from bi-orderable groups. See \cite[Proposition 7]{DDHPV} and \cite[Proposition 3.4]{bps}.

\begin{proposition}\label{prop2}
The following hold for any bi-orderable group $G$:
\begin{enumerate}
\item $\Conj_n(G)$ is a right-orderable quandle.
\item $\Core(G)$ is a left-orderable quandle.
\item If $\phi\in \Aut(G)$ is an order reversing automorphism, then $\Alex(G,\phi)$ is a left-orderable quandle.
\end{enumerate}
\end{proposition}

The other sided orderability of these quandles fails in general \cite[Corollaries 3.8 and 3.9]{bps}.

\begin{proposition}
The following hold for any non-trivial group $G$:
\begin{enumerate}
\item The quandle $\Conj_n(G)$ is not left-orderable.
\item The quandle $\Core(G)$ is not right-orderable.
\item If $\phi\in \Aut(G)$ an involution, then the quandle $\Alex(G,\phi)$ is not right-orderable.
\end{enumerate}
\end{proposition}

An immediate consequence of Proposition \ref{prop2} is the following.

\begin{corollary}
The following hold for any quandle $Q$:
\begin{enumerate}
\item If $Q$ is a subquandle of $\Conj_n(G)$ for some bi-orderable group $G$, then $Q$ is right-orderable. 
\item If $Q$ is a subquandle of $\Core(G)$ for some bi-orderable group $G$, then $Q$ is left-orderable. 
\end{enumerate}
\end{corollary}

\medskip

\section{Properties of linear orderings on quandles}\label{properties-orderable-quandle}
In this section, we analyze some basic properties of linear orderings on quandles. Observe that a quandle essentially has two binary operations $*$ and $*^{-1}$. Thus, it is necessary to understand behaviour of a linear order with respect to both of these binary operations.

\begin{definition}
Let $<$ be a linear order on a quandle $Q$ and $\mathcal{O}$ be the set $\left\{=,<,>\right\}$. For a quadruple $(\blackdiamond_1,\blackdiamond_2,\blackdiamond_3,\blackdiamond_4)\in\mathcal{O}^4$, the order $<$ is said to be of {\it type} $(\blackdiamond_1,\blackdiamond_2,\blackdiamond_3,\blackdiamond_4)$ if the following hold for $x,y,z\in Q$ with $x<y$:
\begin{enumerate}[(1)]
\item $x*z\,\blackdiamond_1\, y*z$,
\item $x*^{-1}z\,\blackdiamond_2\, y*^{-1}z$,
\item $z*x\,\blackdiamond_3\, z*y$,
\item $z*^{-1}x\,\blackdiamond_4\, z*^{-1}y$.
\end{enumerate}
We say that the order $<$ is of type $(\underline{\;\;},\blackdiamond_2,\underline{\;\;},\underline{\;\;})$ if the second condition is true, it is of type $(\blackdiamond_1,\underline{\;\;},\blackdiamond_3,\underline{\;\;})$ if the first and third conditions are true, it is of type $(\blackdiamond_1,\blackdiamond_2,\underline{\;\;},\blackdiamond_4)$ if the first, second and fourth conditions are true, etc.
\end{definition}

The axiom {\textbf Q2} implies that if $<$ is a linear order on a quandle $Q$, then 
\begin{equation}\label{eq17}
x*z\neq y*z\qquad\textrm{and}\qquad x*^{-1}z\neq y*^{-1}z
\end{equation}
for all $x,y,z\in Q$ with $x<y$.

\begin{lemma}\label{lem2}
Let $<$ be a linear order on a quandle $Q$ and $\blackdiamond\in\{<,>\}$. Then the order $<$ is of type $(\blackdiamond,\underline{\;\;},\underline{\;\;},\underline{\;\;})$ if and only if it is of type $(\underline{\;\;},\blackdiamond,\underline{\;\;},\underline{\;\;})$.
\end{lemma}

\begin{proof}
Let $\blackdiamond\in\{<,>\}$. Define $\blackdiamond^{-1}$ to be $>$ if $\blackdiamond$ is $<$ and define it as $<$ if $\blackdiamond$ is $>$. Furthermore, define $\blackdiamond^1$ as $\blackdiamond$. By \eqref{eq17}, we note that $x*z\,\blackdiamond^d\,y*z$ and $x*^{-1}z\,\blackdiamond^e\,y*^{-1}z$ for some $d,e\in\{-1,1\}$ whenever $x,y,z\in Q$ and $x<y$.\\
For the forward implication, suppose on the contrary that $x*^{-1}z\,\blackdiamond^{-1}\,y*^{-1}z$ for some $x,y,z\in Q$ with $x<y$. This implies that $\left(x*^{-1}z\right)*z\,\blackdiamond^{-1}\,\left(y*^{-1}z\right)*z$ if $\blackdiamond$ is $<$ and $\left(x*^{-1}z\right)*z\,\blackdiamond\,\left(y*^{-1}z\right)*z$ if $\blackdiamond$ is $>$, since the order $<$ is of type $(\blackdiamond,\underline{\;\;},\underline{\;\;},\underline{\;\;})$. In other words, $\left(x*^{-1}z\right)*z>\left(y*^{-1}z\right)*z$, that is, $x>y$, which is a contradiction.\\
For the backward implication, suppose on the contrary that $x*z\,\blackdiamond^{-1}\,y*z$ for some $x,y,z\in Q$ with $x<y$. This implies that $\left(x*z\right)*^{-1}z\,\blackdiamond^{-1}\,\left(y*z\right)*^{-1}z$ if $\blackdiamond$ is $<$ and $\left(x*z\right)*^{-1}z\,\blackdiamond\,\left(y*z\right)*^{-1}z$ if $\blackdiamond$ is $>$, since the order $<$ is of type $(\underline{\;\;},\blackdiamond,\underline{\;\;},\underline{\;\;})$. This gives $\left(x*z\right)*^{-1}z>\left(y*z\right)*^{-1}z$, that is, $x>y$, which is a contradiction.
\end{proof}

\begin{lemma}\label{lem3}
Let $<$ be a linear order on a quandle $Q$. Then the following statements are equivalent:
\begin{enumerate}[(1)]
\item The quandle $Q$ is trivial.
\item The order $<$ is of type $(\underline{\;\;},\underline{\;\;},=,\underline{\;\;})$.
\item The order $<$ is of type $(\underline{\;\;},\underline{\;\;},\underline{\;\;},=)$.
\end{enumerate}
\end{lemma}

\begin{proof}
It is trivial that (1) $\Rightarrow$ (2) and (1) $\Rightarrow$ (3).\\
For (2) $\Rightarrow$ (1), let $x$ and $y$ be any elements in $Q$. If $x=y$, then by idempotency, $x*y=x$. If $x<y$ or $y<x$, then by (2), $x*y=x*x=x$. This proves that $x*y=x$ for all $x,y\in Q$.\\
For (3) $\Rightarrow$ (1), let $x$ and $y$ be any elements in $Q$. If $x=y$, then by idempotency, $x*y=x$. If $x<y$ or $y<x$, then by (3), $x*^{-1}y=x*^{-1}x=x$. This implies that $(x*^{-1}y)*y=x*y$, and thus by cancellation, we get $x*y=x$. This proves that $x*y=x$ for all $x,y\in Q$.
\end{proof}

\begin{theorem}\label{th2}
Let $<$ be a linear order on a quandle $Q$ of type $(\blackdiamond_1,\blackdiamond_2,\blackdiamond_3,\blackdiamond_4)$ for some $(\blackdiamond_1,\blackdiamond_2,\blackdiamond_3,\blackdiamond_4)\in \mathcal{O}^4$. Then the following hold:
\begin{enumerate}[(1)]
\item $\blackdiamond_1,\blackdiamond_2\in\left\{<,>\right\}$.
\item $\blackdiamond_1$ and $\blackdiamond_2$ are the same.
\item The quandle $Q$ is trivial $\Leftrightarrow$ $\blackdiamond_3$ is the equality \textquoteleft\,$=$\textquoteright\, $\Leftrightarrow$ $\blackdiamond_4$ is the equality \textquoteleft\,$=$\textquoteright\,.
\item The quadruple $(\blackdiamond_1,\blackdiamond_2,\blackdiamond_3,\blackdiamond_4)$ is one of the following:\\
$(<,<,=,=)$,\quad $(<,<,<,>)$,\quad $(<,<,>,<)$\quad{or}\quad $(>,>,<,<)$.
\end{enumerate}
\end{theorem}

\begin{proof}
Assertion (1) follows from \eqref{eq17}, assertion (2) follows from Lemma \ref{lem2}, and assertion (3) follows from Lemma \ref{lem3}. 
\par

If $\blackdiamond_3$ or $\blackdiamond_4$ is the equality \textquoteleft\,$=$\textquoteright, then by (3), the quandle $Q$ is trivial. In this case, $(\blackdiamond_1,\blackdiamond_2,\blackdiamond_3,\blackdiamond_4)$ must be $(<,<,=,=)$. If $\blackdiamond_3,\blackdiamond_4\in\left\{<,>\right\}$, then by (1) and (2), $(\blackdiamond_1,\blackdiamond_2,\blackdiamond_3,\blackdiamond_4)$ must be one of the following quadruples:
\begin{align*}
\textrm{(a)}\;\,&(<,<,<,<),&\textrm{(b)}\;\,&(<,<,<,>),&\textrm{(c)}\;\,&(<,<,>,<),&\textrm{(d)}\;\,&(<,<,>,>),\\
\textrm{(e)}\;\,&(>,>,<,<),&\textrm{(f)}\;\,&(>,>,<,>),&\textrm{(g)}\;\,&(>,>,>,<),&\textrm{(h)}\;\,&(>,>,>,>).
\end{align*}
To prove assertion (4), we have to rule out the cases (a), (d), (f), (g) and (h). Let us begin by ruling out the case (g) first. Assume contrary that $(\blackdiamond_1,\blackdiamond_2,\blackdiamond_3,\blackdiamond_4)=(>,>,>,<)$. Let $x,y,z\in Q$ such that $x<y$. Then we have
\begin{align}
&z*x>z*y&&(\textrm{since}\,\blackdiamond_3\,\textrm{is}\,>)\\
\Rightarrow\quad&(z*x)*^{-1}x<(z*y)*^{-1}x&&(\textrm{since}\,\blackdiamond_2\,\textrm{is}\,>)\\
\Rightarrow\quad&z<z*y*^{-1}x. &&(\textrm{by right cancellation})\label{eq1}
\end{align}
Furthermore, we have
\begin{align}
&(z*y)*^{-1}x<(z*y)*^{-1}y&&(\textrm{since}\,\blackdiamond_4\,\textrm{is}\,<)\\
\Rightarrow\quad&z*y*^{-1}x<z. &&(\textrm{by right cancellation})
\end{align}
This is a contradiction to \eqref{eq1}. The cases (a), (d) and (f) can be ruled out similarly.
\par

Finally, we rule out the case (h). Assume contrary that $(\blackdiamond_1,\blackdiamond_2,\blackdiamond_3,\blackdiamond_4)=(>,>,>,>)$. Let $x,y,z\in Q$ such that $x<y$. Then we have
\begin{align}
&z*x>z*y&&(\textrm{since}\,\blackdiamond_3\,\textrm{is}\,>)\\
\Rightarrow\quad&x*^{-1}(z*x)<x*^{-1}(z*y)&&(\textrm{since}\,\blackdiamond_4\,\textrm{is}\,>)\\
\Rightarrow\quad&x*^{-1}z*x<x*^{-1}y*^{-1}z*y. &&(\textrm{by Lemma \ref{lem4}})\label{eq2}
\end{align}
Furthermore, since $\blackdiamond_3$ is $>$, we have
\begin{align}
&\left(x*^{-1}z\right)*x>\left(x*^{-1}z\right)*y.&&\label{eq3}
\end{align}
Combining \eqref{eq2} with \eqref{eq3}, we get
\begin{align}
&x*^{-1}z*y<x*^{-1}y*^{-1}z*y&&\\
\Rightarrow\quad&\left(x*^{-1}z*y\right)*^{-1}y>\left(x*^{-1}y*^{-1}z*y\right)*^{-1}y&&(\textrm{since}\,\blackdiamond_2\,\textrm{is}\,>)\\
\Rightarrow\quad&x*^{-1}z>x*^{-1}y*^{-1}z&&(\textrm{by right cancellation})\\
\Rightarrow\quad&\left(x*^{-1}z\right)*z<\left(x*^{-1}y*^{-1}z\right)*z&&(\textrm{since}\,\blackdiamond_1\,\textrm{is}\,>)\\
\Rightarrow\quad&x<x*^{-1}y.&&(\textrm{by right cancellation})\label{eq4}
\end{align}
Since $\blackdiamond_4$ is $>$, we also have
\begin{align}
&x=x*^{-1}x>x*^{-1}y,&&
\end{align}
which is a contradiction to \eqref{eq4}.
\end{proof}

\begin{corollary}
Let $<$ be a linear order on a quandle $Q$. Then the quandle $Q$ is trivial if and only if the order $<$ is of the type $(<,<,=,=)$.
\end{corollary}

We remark that all the four possibilities for the quadruple $(\blackdiamond_1,\blackdiamond_2,\blackdiamond_3,\blackdiamond_4)$ in Theorem \ref{th2} (4) can be realized as we shall see in the following example.

\begin{example}\label{ex1}
Consider the additive group $(\mathbb{R},+)$ of real numbers. For a non-zero $u\in\mathbb{R}$, let $\phi_u$ be the automorphism of $\mathbb{R}$ given by $\phi_u(x)=ux$. Then for the Alexander quandle $\Alex(\mathbb{R},\phi_u)$, the quandle operation $*$ and the dual operation $*^{-1}$ are given by $x*y=ux+(1-u)y$ and $x*^{-1} y=u^{-1}x+\left(1-u^{-1}\right)y$ for $x,y\in \Alex(\mathbb{R},\phi_u)$. With the usual linear order $<$ on $\mathbb{R}$, one can check the following:
\begin{itemize}
\item If $0<u<1$, then $<$ is a bi-ordering for $\Alex(\mathbb{R},\phi_u)$.
\item If $u<1$, then $<$ is a left-ordering for $\Alex(\mathbb{R},\phi_u)$.
\item If $0<u$, then $<$ is a right-ordering for $\Alex(\mathbb{R},\phi_u)$.
\end{itemize}
Further, the following properties of the ordering on $\Alex(\mathbb{R},\phi_u)$ can be checked easily.
\begin{itemize}
\item If $u=1$, then the order $<$ is of the type $(<,<,=,=)$.
\item If $0<u<1$, then the order $<$ is of the type $(<,<,<,>)$.
\item If $1<u$, then the order $<$ is of the type $(<,<,>,<)$.
\item If $u<0$, then the order $<$ is of the type $(>,>,<,<)$.
\end{itemize}
\end{example}

\begin{remark}
Question 3.6 in \cite{bps} asks whether there exists an infinite non-commutative bi-orderable quandle. One can see that for $u\in(0,1)\setminus\left\{1/2\right\}$, the quandle $\Alex(\mathbb{R},\phi_u)$ with the usual order $<$ on $\mathbb{R}$ is an infinite non-commutative bi-orderable quandle, thereby answering the question in an affirmative.
\end{remark}

\begin{proposition}\label{prop3}
Let $<$ be a linear order on a quandle $Q$. Then the order $<$ is a bi-ordering on $Q$ if and only if it is of the type $(<,<,<,>)$.
\end{proposition}

\begin{proof}
It is trivial that if the ordering $<$ is of the type $(<,<,<,>)$, then it is a bi-ordering on $Q$. Conversely, suppose that $<$ is a bi-ordering on $Q$. Then we can say that the ordering is of type $(<,\underline{\;\;},<,\underline{\;\;})$. It follows from Lemma \ref{lem2} that the ordering $<$ is of the type $(<,<,<,\underline{\;\;})$. Now, suppose on the contrary that $<$ is not of the type $(<,<,<,>)$. Then $z*^{-1}x<z*^{-1}y$\, for some $x,y,z\in Q$ with $x<y$. Since $<$ is a left ordering on $Q$, we have
\begin{align}
&\left(z*^{-1}y\right)*x<\left(z*^{-1}y\right)*y,
\end{align}
which using right cancellation yields 
\begin{align}\label{eq16}
\Rightarrow\quad&z*^{-1}y*x<z. 
\end{align}
Also, $<$ being a right ordering gives
\begin{align}
&\left(z*^{-1}x\right)*x<\left(z*^{-1}y\right)*x,
\end{align}
which by right cancellation gives
\begin{align}
\Rightarrow\quad&z<z*^{-1}y*x.
\end{align}
But, this contradicts \eqref{eq16}.
\end{proof}

\begin{proposition}\label{prop10}
Let $<$ be a bi-ordering on a quandle $Q$. If $x,y \in Q$ are distinct elements, then
\begin{enumerate}[(1)]
\item $x*^{-1}y\,\blackdiamond\, x\,\blackdiamond\, x*y\,\blackdiamond\, y\,\blackdiamond\, y*^{-1}x$\quad{and}
\item $x*^{-1}y\,\blackdiamond\, x\,\blackdiamond\, y*x\,\blackdiamond\, y\,\blackdiamond\, y*^{-1}x$
\end{enumerate}
for some $\blackdiamond\in\{<,>\}$.
\end{proposition}

\begin{proof}
Since $x \neq y$, we have $x\,\blackdiamond\, y$ for some $\blackdiamond\in\{<,>\}$. By Proposition \ref{prop3} and axiom {\textbf Q1}, we have
\begin{align*}
\textrm{(a)}\;\,&x=x*^{-1}x\,\blackdiamond^{-1}\, x*^{-1}y,&\textrm{(b)}\;\,&x=x*x\,\blackdiamond\, x*y,&\textrm{(c)}\;\,&x*y\,\blackdiamond\, y*y=y,\\
\textrm{(d)}\;\,&y*^{-1}x\,\blackdiamond^{-1}\, y*^{-1}y=y,&\textrm{(e)}\;\,&x=x*x\,\blackdiamond\, y*x,&\textrm{(f)}\;\,&y*x\,\blackdiamond\, y*y=y.
\end{align*}
Combining the preceding inequalities give the desired result.
\end{proof}

\begin{corollary}
A quasi-commutative bi-orderable quandle is commutative.
\end{corollary}

\begin{proof}
Let $Q$ be a quasi-commutative quandle that is not commutative and $<$ be a bi-ordering on $Q$. Then there exist distinct elements $x$ and $y$ in $Q$ such that at least one of the following hold: $x*y=y*^{-1}x$, $x*^{-1}y=y*x$ or $x*^{-1}y=y*^{-1}x$. By Proposition \ref{prop10}, $x*y\,\,\blackdiamond\, y*^{-1}x$, $x*^{-1}y\,\blackdiamond\, y*x$ and $x*^{-1}y\,\blackdiamond\, y*^{-1}x$ for some $\blackdiamond\in\{<,>\}$. This is a contradiction.
\end{proof}

\medskip

\section{Constructions of orderable quandles and order-preserving automorphisms}\label{construction-order-quandles-automorphisms}
An \textit{action} of a quandle $Q$ on a quandle $X$ is a quandle homomorphism $$\phi : Q \to \Conj_{-1}\big(\Aut(X)\big),$$ where $\Aut(X)$ is the group of quandle automorphisms of $X$, and the operation in $\Conj_{-1}\big(\Aut(X)\big)$ is nothing but $x*y=yxy^{-1}$. Viewing any set $X$ as a trivial quandle, we have $\Aut(X)=\Sigma_X$, the symmetric group on $X$, and we obtain the definition of an action of a quandle $Q$ on a set $X$.

\begin{example}
Some basic examples of quandle actions are:
\begin{itemize}
\item If $Q$ is a quandle, then the map $\phi: Q \to \Conj_{-1} \big(\Aut(Q)\big)$ given by $q \mapsto S_q$ is a quandle homomorphism. Thus, every quandle acts on itself by inner automorphisms.
\item Let $G$ be a group acting on a set $X$. That is, there is a group homomorphism $\phi: G \to \Sigma_X$. Viewing both $G$ and $\Sigma_X$ as conjugation quandles and observing that a group homomorphism is also a quandle homomorphism between corresponding conjugation quandles, it follows that the quandle $\Conj_{-1}(G)$ acts on the set $X$.
\end{itemize}
\end{example}

\begin{theorem}\label{th5}
If a semi-latin quandle is right-orderable, then it acts faithfully on a linearly ordered set by order-preserving bijections. Conversely, if a quandle acts faithfully on a well-ordered set by order-preserving bijections, then it is right-orderable.
\end{theorem}

\begin{proof}
Let $Q$ be a semi-latin quandle that is right-ordered with respect to a linear order $<$. Taking $X=Q$ and defining $\phi: Q \to \Conj_{-1}(\Sigma_X)$ by $\phi(q)=S_q$, we see that $\phi$ is an action of $Q$ on the ordered set $X$. Further, if $q \in Q$ and $x, y \in X$ such that $x<y$, then right-orderability of $Q$ implies that
$$\phi(q)(x)=S_q(x)=x*q < y*q=S_q(y)=\phi(q)(y).$$
Further, if $p, q \in Q$ such that $\phi(p)=\phi(q)$, then $Q$ being semi-latin implies that $p=q$. Hence, $Q$ acts faithfully on $X$ by order-preserving bijections.
\par

Conversely, suppose that $\phi: Q \to \Conj_{-1}(\Sigma_X)$ is a faithful action of $Q$ on a well-ordered set $X$ by order-preserving bijections. Let $<$ be the well-order on $X$. We use the order $<$ to define an order on the quandle $Q$ as follows. For $p, q \in Q$ with $p \neq q$, consider the set $A_{p, q}=\{x \in X~|~\phi(p)(x) \neq \phi(q)(x) \}$. Faithfulness of the action implies that $\phi(p) \neq \phi(q)$, and hence $A_{p, q}$ is a non-empty subset of $X$. Since $<$ is a well-ordering on $X$, the set $A_{p, q}$ must admit the smallest element, say $x_0$, with respect to $<$. We define $p\prec q$ if $\phi(p)(x_0) < \phi(q)(x_0)$ and $q\prec p$ if $\phi(q)(x_0) < \phi(p)(x_0)$. 
\par

It is enough to check transitivity of $\prec$ on $Q$. Let $p \prec q$ and $ q \prec r$. Let $A_{p, q} = \{ x \in X ~|~ \phi(p)(x) \neq \phi(q)(x) \}$, $A_{q, r} = \{ x \in X ~|~ \phi(q)(x) \neq \phi(r)(x) \}$ and $A_{p, r} = \{ x \in X ~|~ \phi(p)(x) \neq \phi(r)(x) \}$. Since $p \prec q$ and $ q \prec r$, it follows that $A_{p, r}$ is non-empty. Let $x_0$, $y_0$ and $z_0$ be the smallest elements of the sets $A_{p, q}$, $A_{q, r}$ and $A_{p, r}$, respectively. Then we have the cases:
\begin{itemize}
\item If $x _0 < y_0$, then $\phi(p)(x_0) < \phi(q)(x_0)=\phi(r)(x_0)$, which implies that $z_0 \leq x_0$. If $z_0 < x_0$, then $\phi(p)(z_0)=\phi(q)(z_0) \neq \phi(r)(z_0)$, which contradicts the fact that $y_0$ is the smallest element of $A_{q, r} $. Hence, $x_0=z_0$ and $ p \prec r$.
\item If $x_0 = y_0$, then $\phi(p)(x_0) < \phi(q)(x_0)< \phi(r) (x_0)$, which gives $z_0 \leq x_0$. If $z_0 < x_0$ , then $\phi(p)(z_0)=\phi(q)(z_0) \neq \phi(r)(z_0)$, which is a contradiction to the fact that $y_0$ is the smallest element of $A_{q, r}$. Hence, $z_0=x_0$ and $p \prec r$.
\item If $x_0 > y_0$, then $\phi(p)(y_0)=\phi(q)(y_0) < \phi(r)(y_0)$, which further gives $z_0 \leq y_0$. If $z_0 < y_0$, then $\phi(p)(z_0)=\phi(q)(z_0) \neq \phi(r)(z_0)$, which is again a contradiction to the fact that $y_0$ is the smallest element of $A_{q, r}$. Hence, $z_0=y_0$ and $ p \prec r$.
\end{itemize}
Now, suppose that $p, q, r \in Q$ such that $p \prec q$. Let $A_{p, q}=\{x \in X~|~\phi(p)(x) \neq \phi(q)(x) \}$ and $A_{p*r, q*r}=\{x \in X~|~\phi(p*r)(x) \neq \phi(q*r)(x) \}$. Since both $A_{p, q}$ and $A_{p*r, q*r}$ are non-empty, they admit smallest elements, say $x_0$ and $y_0$, respectively. Notice that the bijection $\phi(r)$ maps $A_{p, q}$ onto $A_{p*r, q*r}$. Since $\phi(r)$ is order-preserving with respect to $<$, we have $\phi(r)(x_0)=y_0$. Since $p \prec q$, we have $\phi(p)(x_0) < \phi(q)(x_0)$, which implies that $\phi(p)\phi(r)^{-1}(y_0) < \phi(q)\phi(r)^{-1}(y_0)$. Since $\phi$ is a quandle homomorphism and $\phi(r)$ is order-preserving, this gives 
\begin{eqnarray*}
\phi(p*r)(y_0)&=& \phi(p)*\phi(r)(y_0)\\
&=& \phi(r)\phi(p)\phi(r)^{-1}(y_0) \\
&<& \phi(r)\phi(q)\phi(r)^{-1}(y_0)\\
&=& \phi(q)*\phi(r)(y_0)\\
&=& \phi(q*r)(y_0),
\end{eqnarray*}
and hence $p*r \prec q*r$. Thus, $Q$ is a right orderable quandle.
\end{proof}

Next, we give three constructions of orderable quandles. 

\begin{proposition}\label{prop6}
Let $(Q_1,*)$ and $(Q_2,\circ)$ be right-orderable quandles, and $\sigma: Q_1 \to {\rm Conj}_{-1} \left({\rm Aut}(Q_2) \right)$ and $\tau: Q_2 \to {\rm Conj}_{-1} \left({\rm Aut}(Q_1) \right)$ be order-preserving quandle actions. Suppose that
\begin{enumerate}
\item $\tau(z)(x)* y=\tau\left(\sigma(y)(z)\right)(x* y)$ for $x, y \in Q_1$ and $z \in Q_2$,
\item $\sigma(z)(x)\circ y=\sigma\left(\tau(y)(z)\right)(x\circ y)$ for $x, y \in Q_2$ and $z \in Q_1$.
\end{enumerate}
Then $Q=Q_1 \sqcup Q_2$ with the operation
$$
x\star y=\begin{cases}
x*y,& x, y \in Q_1, \\
x\circ y, &x, y \in Q_2, \\
{\tau(y)}(x), &x \in Q_1, y \in Q_2, \\
{\sigma(y)}(x), &x \in Q_2, y \in Q_1,
\end{cases} 
$$
is a right-orderable quandle.
\end{proposition}

\begin{proof}
That $Q$ is a quandle follows from \cite[Proposition 11]{BardakovNasybullovSingh}. Let $<_1$ and $<_2$ be the right-orders on $Q_1$ and $Q_2$, respectively. Define an order $<$ on $Q$ by setting $x< y$ iff $x, y \in Q_1$ and $x<_1 y$ or $x, y \in Q_2$ and $x<_2 y$ or $x\in Q_1$ and $y \in Q_2$. A direct check shows that $<$ is indeed a linear order on $Q$. We claim that $<$ turns $Q$ into a right orderable quandle. Let $x, y, z \in Q$ such that $x<y$. We have the following cases:
\begin{itemize}
\item $x, y, z \in Q_1$ or $x, y,z \in Q_2$: In this case, since $Q_1$ and $Q_2$ are right-orderable, we get $x\star z < y\star z$.
\item $x, y \in Q_1$ and $z \in Q_2$: In this case, since $\tau(z)$ is order preserving, we have $x \star z=\tau(z)(x) <_1 \tau(z)(y)=y \star z$, and hence $x \star z < y \star z$.
\item $x, y \in Q_2$ and $z \in Q_1$: In this case, $\sigma(z)$ being order preserving implies that $x \star z=\sigma(z)(x) <_2 \sigma(z)(y)=y \star z$, and hence $x \star z < y \star z$.
\item $x, z \in Q_1$ and $y \in Q_2$: In this case, $x \star z= x*z \in Q_1$ and $y \star z= \sigma(z)(y) \in Q_2$, and hence $x \star z < y \star z$.
\item $x \in Q_1$ and $y, z \in Q_2$: In this case, $x \star z= \tau(z)(x) \in Q_1$ and $y \star z= y \circ z \in Q_2$, and hence $x \star z < y \star z$.
\end{itemize}
Thus, $Q$ is a right-orderable quandle.
\end{proof}

If $\sigma: Q_1 \to \id_{Q_2}$ and $\tau: Q_2 \to \id_{Q_1}$ are the trivial actions, then conditions (1) and (2) of Proposition \ref{prop6} always hold. Thus, the disjoint union of two right-orderable quandles is right-orderable. It is clear that the disjoint union of two left-orderable quandles is not left-orderable.
\par

Let $\{Q_i, *_i\}_{i \in \Lambda}$ be a family of quandles and $Q= \prod_{i \in \Lambda} Q_i$ their cartesian product. Then $Q$ is a quandle with $(x_i)\star (y_i)= (x_i *_i y_i)$ and called the {\it product quandle}. The following observation is rather immediate, but we include a proof for the sake of completeness.

\begin{proposition}\label{prop1}
The product of right-orderable quandles is a right-orderable quandle. Similarly, the product of bi-orderable quandles is bi-orderable.
\end{proposition}

\begin{proof}
Let $\{Q_i, *_i\}_{i \in \Lambda}$ be a family of right-orderable quandles. Let $<_i$ be the right-order on $Q_i$ for $i \in \Lambda$ and $Q$ their product quandle. By axiom of choice, we can take a well-ordering $<$ on the indexing set $\Lambda$. Let $(x_i), (y_i) \in Q$ such that $(x_i) \neq (y_i)$. Then there exists the least index $\ell \in \Lambda$ such that $x_\ell \neq y_\ell$. We define $(x_i) \prec (y_i)$ if $x_\ell <_\ell y_\ell$ and $(y_i) \prec (x_i)$ if $y_\ell <_\ell x_\ell$. It is easy to check that $\prec$ is a linear order on $Q$.
\par

Let $(x_i), (y_i), (z_i) \in Q$ such that $(x_i) \prec (y_i)$. Then we have $x_\ell <_\ell y_\ell$, where $\ell$ is the least index such that $x_\ell \neq y_\ell$. The second quandle axiom in $Q$ implies that $(x_i *_i z_i)= (x_i)\star(z_i) \neq (y_i) \star (z_i)=(y_i *_i z_i)$. It turns out that $\ell$ is also the least index for which $x_\ell *_\ell z_\ell \neq y_\ell *_\ell z_\ell$. Since $x_\ell <_\ell y_\ell$ and $Q_\ell$ is right orderable, it follows that $x_\ell *_\ell z_\ell <_\ell y_\ell *_\ell z_\ell$. By definition of $\prec$, we have $ (x_i)\star(z_i) \prec (y_i) \star (z_i)$. Thus, $Q$ is a right-orderable quandle. The second assertion follows analogously.
\end{proof}

Let $Q$ be a quandle and $A$ a set. Following \cite[Section 2.1]{Andruskiewitsch2003}, a {\it dynamical 2-cocycle} is a map $\alpha: Q \times Q \to \Map(A \times A, A)$ such that 
\begin{equation}\label{dynamical-cocycle-condition1}
\alpha_{x, x}(s, s)=s,
\end{equation}
\begin{equation}\label{dynamical-cocycle-condition2}
\alpha_{x, y}(-, t): A \to A~\textrm{is a bijection}
\end{equation}
and the {\it cocycle condition}
\begin{equation}\label{dynamical-cocycle-condition3}
\alpha_{x*y, z}\big(\alpha_{x, y}(s, t), ~u \big)= \alpha_{x* z, y*z}\big(\alpha_{x, z}(s, u),~ \alpha_{y, z}(t, u) \big)
\end{equation}
holds for all $x, y, z \in Q$ and $s, t, u \in A$. Given a dynamical 2-cocycle $\alpha$, the set $Q \times A$ can then be turned into a quandle denoted as $Q \times_{\alpha} A$ by defining
\begin{equation}\label{dynamical-quandle-operation}
(x, s)* (y,t)= \big( x* y, ~\alpha_{x, y}(s, t) \big).
\end{equation} 
The equations \eqref{dynamical-cocycle-condition1}, \eqref{dynamical-cocycle-condition2} and \eqref{dynamical-cocycle-condition3} give the quandle axioms {\textbf Q1}, {\textbf Q2} and {\textbf Q3}, respectively. The quandle $Q \times_{\alpha} A$ is called the {\it extension} of $Q$ by $A$ through $\alpha$.
\par
If $A$ is an abelian group, then a normalized quandle 2-cocycle is a map $\alpha: Q \times Q \to A$ satisfying
$$
\alpha_{x, y}~\alpha_{x*y, z}= \alpha_{x, z}~\alpha_{x* z, y*z}
$$
and
$$
\alpha_{x, x}=1
$$
for all $x, y, z \in Q$. A normalized quandle 2-cocycle $\alpha: Q \times Q \to A$ gives rise to a dynamical 2-cocycle $\alpha': Q \times Q \to \Map(A \times A, A)$ defined as $$\alpha'_{x, y}(s, t)=s~\alpha_{x, y}.$$ In this case, the quandle $X \times_\alpha A$ is called the {\it abelian extension} of $Q$ by $A$ through $\alpha$. Such extensions appeared first in \cite{CKS03}.
\par

\begin{proposition}\label{orderable-extensions}
The following statements hold:
\begin{enumerate}
\item Let $Q$ be a right-orderable quandle, $A$ an ordered set and $\alpha: Q \times Q \to \Map(A \times A, A)$ a dynamical 2-cocycle. If $\alpha_{x, y}:A \times A \to A$ is order-preserving for all $x, y \in Q$, then the quandle $Q \times_\alpha A$ is right-orderable.
\item If $Q$ is a right-orderable quandle, $A$ a right-orderable abelian group and $\alpha: Q \times Q \to A$ a normalized 2-cocycle, then the quandle $X \times_\alpha A$ is right-orderable.
\item If $Q$ is a quandle, $A$ a non-trivial abelian group and $\alpha: Q \times Q \to A$ a normalized 2-cocycle, then the quandle $X \times_\alpha A$ cannot be left-orderable.
\end{enumerate}
\end{proposition}

\begin{proof}
Let $<$ be a right-order on $Q$ and $<'$ an order on $A$. Consider the set $Q \times A$ with the induced lexicographic order $\prec$ and $A \times A$ equipped with the lexicographic order $\prec'$. Let $(x, s), (y, t), (z, u) \in Q \times A$. By \eqref{dynamical-quandle-operation}, we have $(x, s)*(z, u) =(x*z, \alpha_{x, z}(s, u))$ and $(y, t)*(z, u) =(y*z, \alpha_{y, z}(t, u))$. If $(x, s) \prec (y, t)$, then we have two cases:
\begin{itemize}
\item If $x<y$, then right-orderability of $Q$ implies that $x*z < y*z$, and hence $(x, s)*(z, u) \prec (y, t)*(z, u)$.
\item If $x=y$ and $s <' t$, then $(s, u) \prec' (t, u)$ and $\alpha_{x, z}$ being order-preserving implies that
$\alpha_{x, z}(s, u) <' \alpha_{x, z}(t, u) =\alpha_{y, z}(t, u)$.
\end{itemize}
Hence, $(x, s)*(z, u) \prec (y, t)*(z, u)$, and the quandle $X \times_\alpha A$ is right-orderable proving (1).
\par

Define $\alpha'_{x, y}(s, t)=s~\alpha_{x, y}$ for all $x, y \in Q$ and $s, t \in A$. Then $\alpha'$ is a dynamical 2-cocycle. The right-orderability of the abelian group $A$ implies that $$\alpha'_{x, z}(s, u)=s~\alpha_{x, z} <' t~\alpha_{x, z}= \alpha'_{x, z}(t, u)$$
for all $x, z \in Q$ and $s, t, u \in A$ with $s<' t$. The proof of assertion (2) now follows along the lines of that of assertion (1).
\par
For assertion (3), notice that any left-orderable quandle must be semi-latin. But, for any $(x, s), (x, t), (z, u) \in Q \times A$ with $s \neq t$, we have 
$$(z, u)*(x, s)=(z*x, u~\alpha_{z, x})=(z, u)* (x, t).$$ Hence, the abelian extension $X \times_\alpha A$ cannot be left-orderable.
\end{proof}

We conclude this section with some observations on order-preserving automorphisms of orderable quandles. Let $\Aut^\circ(Q)$ denote the group of order-preserving automorphisms of a quandle $Q$ equipped with an order. Similarly, let $\Aut^\circ(G)$ denote the group of order-preserving automorphisms of a group $G$ equipped with an order.

\begin{proposition}\label{inn-order-auto}
If $Q$ is a right-orderable quandle, then $\Inn(Q)$ is a subgroup of $\Aut^\circ(Q)$.
\end{proposition}

\begin{proof}
Let $<$ be a right-order on $Q$, $x, y \in Q$ with $x<y$ and $S_{z_1}^{d_1}S_{z_2}^{d_2} \cdots S_{z_k}^{d_k} \in \Inn(Q)$, where $d_i \in \{1, -1\}$ and $z_i \in Q$. Then, right-orderability of $Q$ and Lemma \ref{lem2} implies that
$$S_{z_1}^{d_1}S_{z_2}^{d_2} \cdots S_{z_k}^{d_k}(x)= x*^{d_k} z_k*^{d_{k-1}} \cdots *^{d_1} z_1 < y*^{d_k} z_k*^{d_{k-1}} \cdots *^{d_1} z_1=S_{z_1}^{d_1}S_{z_2}^{d_2} \cdots S_{z_k}^{d_k}(y),$$
and hence $\Inn(Q) \le \Aut^\circ(Q)$.
\end{proof}

Note that Proposition \ref{inn-order-auto} fails if $Q$ is a left-orderable quandle. In Example \ref{ex1}, if we take $u<0$, then the quandle $\Alex(\mathbb{R},\phi_u)$ is left-orderable. However, if $x,y, z \in \Alex(\mathbb{R},\phi_u)$ with $x<y$, then $S_z(y) < S_z(x)$.

\begin{proposition}\label{order-aut-conj}
The following hold for any bi-orderable group $G$:
\begin{enumerate}
\item There is an embedding of groups $\Z(G) \rtimes \Aut^\circ(G) \hookrightarrow \Aut^\circ\big(\Conj(G)\big)$.
\item If $G$ has trivial center, then $\Aut^\circ (\Conj(G)) =\Aut^\circ(G)$.
\end{enumerate}
\end{proposition}

\begin{proof}
Let $<$ be a bi-ordering on the group $G$. Then, by Proposition \ref{prop2} (1), $\Conj(G)$ is a right-orderable quandle with respect to $<$. By \cite[Proposition 4.7]{bardakov-dey-singh}, $\Z(G) \rtimes \Aut(G) \hookrightarrow \Aut\big(\Conj(G)\big)$, where each central element $z \in \Z(G)$ act on $\Conj(G)$ by left translation $t_z$ by $z$. Left orderability of $G$ implies that $t_z \in \Aut^\circ (\Conj(G))$ for all $z \in \Z(G)$. Since $\Aut^\circ(G) \le \Aut^\circ (\Conj(G))$, we obtain $\Z(G) \rtimes \Aut^\circ(G) \hookrightarrow \Aut^\circ\big(\Conj(G)\big)$.
\par
It follows from \cite[Corollary 1]{BardakovNasybullovSingh} that if $G$ has trivial center, then $\Aut (\Conj(G))=\Aut(G)$, and hence $\Aut^\circ (\Conj(G)) =\Aut^\circ(G)$.
\end{proof}

\begin{proposition}\label{order-aut-alex}
Let $G$ be a bi-orderable group and $\varphi \in \Aut(G)$ an order-reversing automorphism. Then the following hold:
\begin{enumerate}
\item There is an embedding $\Z(G) \rtimes \C_{\Aut^\circ(G)}(\varphi) \hookrightarrow \Aut^\circ\big(\Alex(G, \varphi)\big)$, where $\C_{\Aut^\circ(G)}(\varphi)=\{f \in \Aut^\circ(G)~|~ f \phi=\phi f \}$.
\item If $G$ is a torsion free abelian group, then $\Aut^\circ\big(\Core(G)\big) \cong G \rtimes \Aut^\circ(G)$. 
\end{enumerate}
\end{proposition}

\begin{proof}
Let $<$ be a bi-ordering on the group $G$ and $\phi\in \Aut(G)$ an order-reversing automorphism. Then, by Proposition \ref{prop2} (3), $\Alex(G,\phi)$ is a left-orderable quandle with respect to the order $<$. Further, by \cite[Proposition 4.1]{bardakov-dey-singh}, $\Z(G) \rtimes \C_{\Aut(G)}(\varphi) \hookrightarrow \Aut\big(\Alex(G, \varphi)\big)$, where an element $z \in \Z(G)$ act on the quandle $\Alex(G,\phi)$ by left translation $t_z$. Since $G$ is left-orderable, the translation $t_z \in \Aut^\circ\big(\Alex(G, \varphi)\big)$ for each $z \in \Z(G)$. Further, if $f \in \C_{\Aut^\circ(G)}(\varphi)$, then $f \in \Aut^\circ\big(\Alex(G, \varphi)\big)$, and hence $\Z(G) \rtimes \C_{\Aut^\circ(G)}(\varphi) \hookrightarrow \Aut^\circ\big(\Alex(G, \varphi)\big)$.
\par
For the second assertion, note that every torsion free abelian group $G$ is bi-orderable. Taking $\varphi(g)=g^{-1}$ for all $g \in G$, we have $\Alex(G, \varphi)=\Core(G)$ and $\C_{\Aut^\circ(G)}(\varphi)=\Aut^\circ(G)$. By \cite[Theorem 4.2]{bardakov-dey-singh}, $\Aut(
\Core(G)) \cong G \rtimes \Aut(G)$, and hence $\Aut^\circ(\Core(G)) \cong G \rtimes \Aut^\circ(G)$.
\end{proof}

\medskip

\section{Orderability of some general quandles}\label{orderability-general-quandle}
In this section, we discuss orderability of some general quandles. The construction of free racks due to Fenn and Rourke \cite[p.351]{frk} and free quandles due to Kamada \cite{k1,k2} has been extended in a recent work of Bardakov and Nasybullov \cite{bn} to what they refer as $(G,A)$-racks/quandles. In fact, many well-known quandles can be seen as $(G, A)$-quandles.
\par

Let $G$ be a group and $A$ be a subset of $G$. Then the set $A \times G$ becomes rack under the following operation
\begin{equation*}
(a,u)*(b,v)= (a, uv^{-1}bv)\;\;\text{for}\;a,b\in A\;\textrm{and}\;u,v\in G.
\end{equation*}
The rack defined as above is known as $(G,A)$-{\it rack} and is denoted by $R(G,A)$. Let $Q(G,A)$ be the quotient of the set $A \times G$ by the equivalence relation $(a, vu) \sim (a,u)$ if and only if $v \in C_G(a)=\{x \in G~|~xa=ax\}$. Denote $[(a,u)]$ to be the equivalence class of $(a,u)$ in $Q(G,A)$. The set $Q(G,A)$ becomes quandle under the following operation
\begin{equation*}
[(a,u)]*[(b,v)]= [(a,uv^{-1}bv)]\;\;\text{for}\;a,b\in A\;\textrm{and}\;u,v\in G,
\end{equation*}
and this quandle is known as $(G,A)$-{\it quandle}. For simplicity, we will write $(a,u)$ instead of $[(a,u)]$ throughout this section.

There is a natural rack homomorphism $\epsilon: R(G,A) \to \Conj(G)$ defined as $\epsilon(a,u)=u^{-1}a u$. Moreover, this map induces a quandle homomorphism $\overline{\epsilon}: Q(G, A) \to \Conj(G)$ defined as $\overline{\epsilon}(a,u)=u^{-1} a u$.

Recall from Fenn and Rourke \cite[p.351]{frk} that the free rack $FR(A)$ on a set $A$ is the rack $R(F(A),A)$, where $F(A)$ is the free group on the set $A$. On the other hand, Kamada \cite{k1,k2} defined the free quandle $FQ(A)$ on a set $A$ as a quotient of $FR(A)$ modulo the equivalence relation generated by $$(a, w)= (a, aw)$$ for $a \in A$ and $w \in F(A)$. Furthermore, if $A$ is the set of representatives of conjugacy classes of a group $G$, then $Q(G, A) \cong \Conj(G)$.

\begin{theorem}\label{th3}
Let $G$ be a group and $A$ be a subset of $G$.
\begin{enumerate}
\item If $G$ is right-orderable, then the rack $R(G,A)$ is right-orderable.
\item If $G$ is bi-orderable, then the quandle $Q(G,A)$ is right-orderable.
\end{enumerate}
\end{theorem}

\begin{proof}
\begin{enumerate}[(1)]
\item Let $<$ be a right ordering on $G$. We define a linear order $<'$ on $R(G,A)$ as follows. Let $(a,u)$ and $(b,v)$ be two distinct elements of $R(G,A)$.
\begin{itemize}
\item If $a \neq b$, define $(a,u)<'(b,v)$ if $a<b$ and $(b,v)<'(a,u)$ if $b<a$.
\item If $a=b$, define $(a,u)<'(a,v)$ if $u<v$ and $(a,v)<'(a,u)$ if $v<u$.
\end{itemize}
Let $(a,u), (b,v), (c,w) \in R(G,A)$ such that $(a,u)<'(b,v)$. If $a \neq b$, then $a<b$, and hence $(a,u)*(c,w)<'(b,v)*(c,w)$. If $a=b$, then $u<v$. Since $G$ is right-ordered with respect to $<$, it follows that $uw^{-1}cw<vw^{-1}cw$, and hence $(a,u)*(c,w)<'(a,v)*(c,w)$. This shows that $R(G,A)$ is a right-orderable rack.
\medskip
\item Let $<$ be a bi-ordering on $G$. Define a linear order $<'$ on $Q(G,A)$ as follows. Let $(a,u)$ and $(b,v)$ be two distinct elements of $Q(G,A)$. 
\begin{itemize}
\item If $a \neq b$, then define $(a,u)<'(b,v)$ if $a<b$ and $(b,v)<'(a,u)$ if $b<a$.
\item If $a=b$, then we define the order using the image of $(a,u)$ and $(a,v)$ under the map $\overline{\epsilon}: Q(G,A) \to G$. Notice that, if $(a,u) \neq (a,v)$ in $Q(G,A)$, then $\overline{\epsilon}(a,u)\neq \overline{\epsilon}(a,v)$. For, if $u^{-1}au=v^{-1}av$, then $vu^{-1}a=a vu^{-1}$; this implies that $vu^{-1} \in C_G(a)$ and hence $(a,u) =(a,v)$. Now, define $(a,u)<'(a,v)$ if $u^{-1} a u<v^{-1} a v$ and $(a,v)<'(a,u)$ if $v^{-1}av<u^{-1}au$.
\end{itemize}
We claim that $Q(G,A)$ is right-ordered with respect to $<'$. Let $(a,u), (b,v), (c,w) \in R(G,A)$ such that $(a,u)<'(b,v)$. If $a \neq b$, then $a<b$, and hence $(a,u)*(c,w)<'(b,v)*(c,w)$. If $a=b$, then $u^{-1}au<v^{-1} a v$. Since $G$ is bi-ordered with respect to $<$, we have $w^{-1}c^{-1}wu^{-1}auw^{-1}cw<w^{-1}c^{-1}wv^{-1}avw^{-1}cw$, and hence $(a,u)*(c,w)<'(a,v)*(c,w)$. This shows that $Q(G,A)$ is right-orderable.\qedhere
\end{enumerate}
\end{proof}

If $A$ is the set of representatives of conjugacy classes of a group $G$, then $Q(G, A)\cong \Conj(G)$. Thus, we recover Proposition \ref{prop2} (1). Further, since free groups are bi-orderable \cite{v}, we retrieve the following result of \cite[Theorem 3.5]{bps}.

\begin{corollary}\label{cor2}
Free quandles are right-orderable. In particular, link quandles of trivial links are right orderable.
\end{corollary}
\medskip

Next, we give a sufficient condition for the failure of left-orderability in quandles.

\begin{proposition}\label{prop11}
Let $Q$ be a quandle generated by a set $X$ such that the map $\eta:Q\to\Conj(\Env(Q))$ is injective. If there exist two distinct commuting elements in $\Env(Q)$ that are not inverses of each other and that are conjugates of elements from $\eta(X)^{\pm1}$, then the quandle $Q$ is not left-orderable.
\end{proposition}

\begin{proof}
Recall from Theorem \ref{th1} that the set $\eta(X)= \{ \tilde{x}~|~x \in X\}$ is a generating set for the enveloping group $\Env(Q)$. Let $\eta(X)^{-1}$ denote the set of inverses of elements in $\eta(X)$, and $\tilde{a},\tilde{b}\in\Env(Q)$ with ${\tilde{a}}^{\,\pm1}\neq\tilde{b}$ be two commuting elements that are conjugates of elements from $\eta(X)^{\pm1}$. Then we can write
\begin{align*}
\tilde{a}&={\tilde{x}_1}^{\,-d_1}{\tilde{x}_2}^{\,-d_2}\cdots{\tilde{x}_{m-1}}^{\,-d_{m-1}}\tilde{x}_m^{\,d_m}\:\!{\tilde{x}_{m-1}}^{\,d_{m-1}}\cdots{\tilde{x}_2}^{\,d_2}{\tilde{x}_1}^{\,d_1}\quad\textrm{and}\\
\tilde{b}&={\tilde{y}_1}^{\,-e_1}{\tilde{y}_2}^{\,-e_2}\cdots{\tilde{y}_{n-1}}^{\,-e_{n-1}}\tilde{y}_n^{\,e_n}\:\!{\tilde{y}_{n-1}}^{\,e_{n-1}}\cdots{\tilde{y}_2}^{\,e_2}{\tilde{y}_1}^{\,e_1}\,,
\end{align*}
where $\tilde{x}_i,\tilde{y}_i\in\eta(X)$ and $d_i,e_i\in\{-1,1\}$ for all $i$. For each $i$, there exist $x_i,y_i\in X$ such that $\tilde{x}_i=\eta(x_i)$ and $\tilde{y}_i=\eta(y_i)$. We get
\begin{align*}
{\tilde{a}}^{\;\!d_m}&=\eta(x_1)^{-d_1}\eta(x_2)^{-d_2}\cdots\eta(x_{m-1})^{-d_{m-1}}\eta(x_m)\:\!\eta(x_{m-1})^{d_{m-1}}\cdots\eta(x_2)^{d_2}\eta(x_1)^{d_1}\\
&=\eta(x_m)*^{d_{m-1}}\eta(x_{m-1})*^{d_{m-2}}\cdots*^{d_1}\eta(x_1),~\textrm{by quandle operation in}~ \Conj(\Env(Q))\\
&=\eta\!\left(x_m*^{d_{m-1}}x_{m-1}*^{d_{m-2}}\cdots*^{d_1}x_1\right),~\textrm{since $\eta$ is a quandle homomorphism}\\
&=\eta(a)
\end{align*}
and similarly
\begin{equation*}
{\tilde{b}}^{\;\!e_n}=\eta(b),
\end{equation*}
where $a=x_m*^{d_{m-1}}x_{m-1}*^{d_{m-2}}\cdots*^{d_1}x_1$ and $b=y_n*^{e_{n-1}}y_{n-1}*^{e_{n-2}}\cdots*^{e_1}y_1$. 
\par

Suppose on the contrary that the quandle $Q$ is left-ordered with respect to a linear order $<$. Since ${\tilde{a}}^{\,\pm1}\neq\tilde{b}$, we get ${\tilde{a}}^{\;\!d_m}\neq{\tilde{b}}^{\;\!e_n}$, and thus $\eta(a)\neq\eta(b)$. This implies that $a\neq b$. In other words, we have $a\,\blackdiamond\, b$ for some $\blackdiamond\in\{<,>\}$, and hence $a=a*a\,\blackdiamond\, a*b$. Since $\tilde{b}^{\,-1}\tilde{a}\tilde{b}=\tilde{a}$, we have $\tilde{b}^{\,-1}{\tilde{a}}^{\;\!d_m}\tilde{b}={\tilde{a}}^{\;\!d_m}$, and thus
\begin{equation*}
\eta(a*^{e_n}b)=\eta(a)*^{e_n}\eta(b)=\eta(b)^{-e_n}\eta(a)\eta(b)^{e_n}=\tilde{b}^{\,-1}{\tilde{a}}^{\;\!d_m}\tilde{b}={\tilde{a}}^{\;\!d_m}=\eta(a).
\end{equation*}
The map $\eta$ being a monomorphism gives $a*^{e_n}b=a$, and hence $a*b=a$. This is a contradiction, since we have $a\,\blackdiamond\, a*b$.
\end{proof}

If $Q$ is a trivial quandle with more than one element, then its enveloping group $\Env(Q)$ is the free abelian group of rank $|Q|$. Thus, if $x, y \in Q$ are two distinct elements, then $\tilde{x}, \tilde{y} \in \Env(Q)$ are two distinct commuting elements that are not inverses of each other. Thus, $Q$ is not left-orderable, which can also be checked directly.

\begin{corollary}
Let $K$ be a prime knot such that $Q(K)$ is generated by a set $X$. If there exist two distinct commuting elements in $G(K)$ that are not inverses of each other and that are conjugates of elements from $\eta(X)^{\pm1}$, then $Q(K)$ is not left-orderable.
\end{corollary}

\begin{proof}
If $K$ is a prime knot, then by \cite[Corollary 3.6]{r}, the map $\eta: Q\left(K\right)\to\Conj(G(K))$ is a monomorphism of quandles. The result now follows from Proposition \ref{prop11}.
\end{proof}

\medskip

\section{Orderability of some link quandles}\label{orderability-knot-quandle}
Problem 3.16 in \cite{bps} asks to determine whether link quandles are orderable. We investigate orderability of link quandles in the remaining two sections and provide a solution to this problem in some cases. The next result relates orderability of the enveloping group of a quandle to that of the quandle itself.

\begin{proposition}\label{prop4}
Let $Q$ be a quandle such that the natural map $\eta:Q\to \Conj(\Env(Q))$ is injective. If $\Env(Q)$ is a bi-orderable group, then $Q$ is a right-orderable quandle.
\end{proposition} 

\begin{proof}
Since $\Env(Q)$ is a bi-orderable group, by Proposition \ref{prop2} (1), $\Conj(\Env(Q))$ is a right-orderable quandle. Since $\eta$ is injective, it follows that $Q$ is right-orderable.
\end{proof} 

\begin{corollary}
If $Q$ is a commutative, latin or simple quandle such that $\Env(Q)$ is a bi-orderable group, then $Q$ is right-orderable.
\end{corollary}

\begin{proof}
It is not difficult to see that the map $\eta$ is injective for a commutative, latin or simple quandle.
\end{proof}

\begin{corollary}\label{cor1}
If the knot group of a prime knot is bi-orderable, then its knot quandle is right-orderable.
\end{corollary}

\begin{proof}
Let $K$ be a prime knot such that its knot group $G(K)$ is bi-orderable. Since $K$ is prime, by \cite[Corollary 3.6]{r}, the map $\eta: Q\left(K\right)\to \Conj(G\left(K\right))$ is injective. Thus, by Proposition \ref{prop4}, the knot quandle $Q\left(K\right)$ is right-orderable.
\end{proof}

\begin{corollary}\label{cor4}
If all the roots of the Alexander polynomial of a fibered prime knot are real and positive, then its knot quandle is right-orderable.
\end{corollary}

\begin{proof}
Let $K$ be a fibered prime knot all the roots of whose Alexander polynomial are real and positive. Then, by \cite[Theorem 1.1]{pr}, $G(K)$ is a bi-orderable group. The result now follows from Corollary \ref{cor1}.
\end{proof}

As a special case, it follows that the knot quandle of the figure eight knot is right-orderable.
\medskip

\begin{definition}
A link $L$ is said to be {\it positive} if there exists a diagram $D$ of $L$ such that all its crossings are positive.
\end{definition}

A diagram $D$ of a link $L$ is said to be
\begin{itemize}
\item {\it minimal} if it is having the minimal number of crossings among all diagrams of $L$.
\item {\it positive} if all its crossings are positive.
\item {\it positive minimal} if it is both positive as well as minimal.
\item {\it minimal positive} if it is positive and having the minimal number of crossings among all positive diagrams of $L$.
\end{itemize}

The terms {\it negative link}, {\it negative diagram}, {\it negative minimal diagram} and {\it minimal negative diagram} are defined analogously.

If a positive minimal diagram exists for a positive link $L$, then it is always a minimal positive diagram of $L$. There are examples of positive links for which positive minimal diagrams do not exist. For example, the number of crossings in a minimal positive diagram of the knot $11_{550}$ is $12$ while its crossing number is $11$. In other words, a positive minimal diagram does not exist for this knot. See \cite{n1,s} for more details.

\begin{theorem}\label{th6}
Let $L_1$ be any link and $L_2$ a non-trivial positive (negative) link. Suppose there exists a minimal positive (negative) diagram $D_2$ of $L_2$ such that generators of the link quandle $Q(L_2)$ corresponding to arcs in $D_2$ are pairwise distinct. Then the link quandle of a connected sum of links $L_1$ and $L_2$ is not bi-orderable. In particular, the link quandle $Q(L_2)$ is not bi-orderable.
\end{theorem}

\begin{proof}
Let $L=L_1\#L_2$ be the link obtained by taking the connected sum of a component $K_1$ of $L_1$ with a component $K_2$ of $L_2$. Suppose $D_1$ be a diagram of $L_1$ such that the component $K_1$ of $L_1$ has an exterior arc in $D_1$, and $D_2$ be a diagram of $L_2$ as described in the hypothesis of the theorem. Let $D$ be a diagram of $L$ obtained using diagrams $D_1$ and $D_2$ without introducing any extra crossing and possibly turning over the diagram $D_1$ if required. The diagram $D$ looks as shown in Figure \ref{fig5} or in Figure \ref{fig6} depending on whether the component $K_2$ of $L_2$ has an exterior arc in $D_2$ or not. In both the figures, the diagram $C_1$ is either $D_1$ or it is obtained by turning over $D_1$.

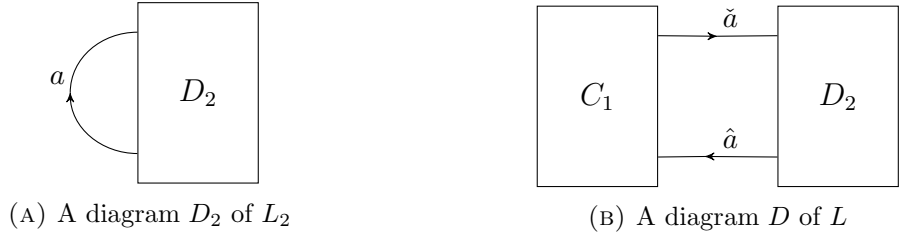
\begin{figure}[H]
\begin{subfigure}{0.45\textwidth}
\centering
\begin{tikzpicture}
\node (B) at (0,0) [draw, minimum width=16mm, minimum height=24mm]{{\large $D_2$}};
\begin{knot}[clip width=6, clip radius=4pt]
\strand[-] (-1.7,0) to [out=up, in=left] (B.135);
\strand[->] (B.225) to [out=left, in=down] (-1.7,0);
\end{knot};
\node at (-1.87,0.2) {$a$};
\end{tikzpicture}
\caption{A diagram $D_2$ of $L_2$}
\end{subfigure}
\begin{subfigure}{0.45\textwidth}
\centering
\begin{tikzpicture}
\node (A) at (-1.6,0) [draw, minimum width=16mm, minimum height=24mm]{{\large $C_1$}};
\node (B) at (1.6,0) [draw, minimum width=16mm, minimum height=24mm]{{\large $D_2$}};
\begin{knot}[clip width=6, clip radius=4pt]
\strand[->] (A.45) to [out=right, in=left] (-0.01,0.8);
\strand[-] (-0.01,0.8) to [out=right, in=left] (B.135);
\strand[->] (B.225) to [out=left, in=right] (-0.17,-0.8);
\strand[-] (-0.17,-0.8) to [out=left, in=right] (A.315);
\end{knot};
\node at (0.17,1.05) {$\check{a}$}; \node at (0.17,-0.55) {$\hat{a}$};
\end{tikzpicture}	
\caption{A diagram $D$ of $L$}
\end{subfigure}
\caption{If the component $K_2$ of $L_2$ has an exterior arc in $D_2$}
\label{fig5}
\end{figure}

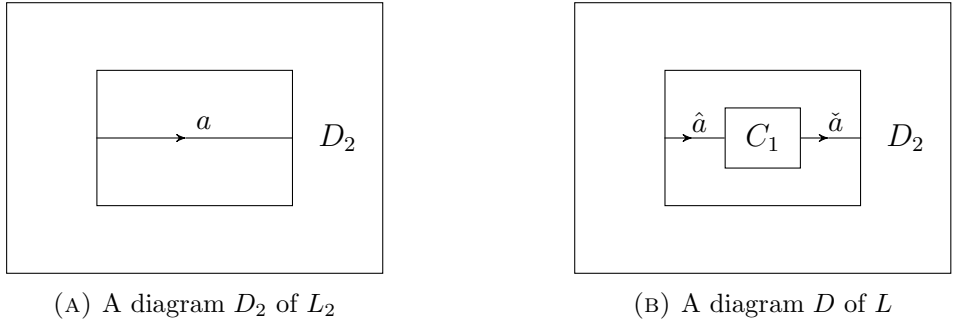
\begin{figure}[H]
\begin{subfigure}{0.45\textwidth}
\centering
\begin{tikzpicture}
\node (B) at (0,0) [draw, minimum width=26mm, minimum height=18mm]{};
\node (C) at (0,0) [draw, minimum width=50mm, minimum height=36mm]{};
\node at (1.9,0) {{\large $D_2$}};
\begin{knot}[clip width=6, clip radius=4pt]
\strand[-] (-0.12,0) to [out=right, in=left] (B.0);
\strand[->] (B.180) to [out=right, in=left] (-0.12,0);
\end{knot};
\node at (0.12,0.2) {$a$};
\end{tikzpicture}
\caption{A diagram $D_2$ of $L_2$}
\end{subfigure}
\begin{subfigure}{0.45\textwidth}
\centering
\begin{tikzpicture}
\node (A) at (0,0) [draw, minimum width=10mm, minimum height=8mm]{{\large $C_1$}};
\node (B) at (0,0) [draw, minimum width=26mm, minimum height=18mm]{};
\node (C) at (0,0) [draw, minimum width=50mm, minimum height=36mm]{};
\node at (1.9,0) {{\large $D_2$}};
\begin{knot}[clip width=6, clip radius=4pt]
\strand[->] (A.0) to [out=right, in=left] (0.87,0);
\strand[-] (0.87,0) to [out=right, in=left] (B.0);
\strand[->] (B.180) to [out=right, in=left] (-0.94,0);
\strand[-] (-0.94,0) to [out=right, in=left] (A.180);
\end{knot};
\node at (0.97,0.22) {$\check{a}$}; \node at (-0.84,0.22) {$\hat{a}$};
\end{tikzpicture}	
\caption{A diagram $D$ of $L$}
\end{subfigure}
\caption{If the component $K_2$ of $L_2$ has no exterior arc in $D_2$}
\label{fig6}
\end{figure}

Let $x_0,x_1,\ldots,x_n$ be generators of the link quandle $Q(L_2)$ corresponding to arcs in $D_2$. We may assume that $x_0$ corresponds to the arc $a$ in $D_2$ that splits into the connecting arcs $\check{a}$ and $\hat{a}$ in $D$. Looking at Figure \ref{fig5} and Figure \ref{fig6}, the arc $\check{a}$ is an incoming arc to $D_2$ and the arc $\hat{a}$ is an outgoing arc from $D_2$. Let $\check{x}_0$ and $\hat{x}_0$ be the elements in the link quandle $Q(L)$ that correspond to the arcs $\check{a}$ and $\hat{a}$ respectively. By the hypothesis of the theorem, the generators $x_0,x_1,\ldots,x_n$ are pairwise distinct in $Q(L_2)$, and thus the elements $\check{x}_0,\hat{x}_0,x_1,x_2,\ldots,x_n$ are pairwise distinct in $Q(L)$ except possibly for the pair $\check{x}_0$ and $\hat{x}_0$.

Suppose on the contrary that the quandle $Q(L)$ is bi-ordered with respect to a linear order $<$. Then we have the smallest and the largest elements in any finite subset of $Q(L)$. Let us consider the following cases:

\begin{enumerate}[(1)]
\item $L_2$ is a positive link: Let $\hat{y}_1$ and $\hat{y}_2$ be the smallest and largest elements, respectively, in the set $\{\hat{x}_0,x_1,x_2,\ldots,x_n\}$. Since $L_2$ is a non-trivial link, we have $n\geq1$, and hence $\hat{y}_1<\hat{y}_2$. For $i=1,2$, consider the crossing $\hat{c}_i$ where the arc corresponding to $\hat{y}_i$ is an outgoing arc (see Figure \ref{fig7}). Note that $\hat{c}_i$ must be a crossing in $D_2$. Let $\hat{u}_i\in\{\check{x}_0,x_1,x_2,\ldots,x_n\}$ and $\hat{v}_i\in\{\hat{x}_0,x_1,x_2,\ldots,x_n\}\cup\{\check{x}_0\}$ be the elements corresponding to the incoming arc and the over arc at $\hat{c}_i$, respectively (see Figure \ref{fig7}). We claim that $\hat{u}_i\neq\hat{v}_i$. Suppose on the contrary that $\hat{u}_i=\hat{v}_i$. Since $\check{x}_0,\hat{x}_0,x_1,x_2,\ldots,x_n$ are pairwise distinct except possibly for the pair $\check{x}_0$ and $\hat{x}_0$, we must have either (a) $\hat{u}_i=\hat{v}_i=x_j$ for some $j$, or (b) $\hat{u}_i=\check{x}_0$ and $\hat{v}_i\in\{\check{x}_0\}\cup\{\hat{x}_0\}$. If $\hat{u}_i=\hat{v}_i=x_j$, then the arc corresponding to $x_j$ is the incoming as well as over arc at $\hat{c}_i$. This contradicts to the fact that $D_2$ is a minimal positive diagram of $L_2$. If $\hat{u}_i=\check{x}_0$ and $\hat{v}_i\in\{\check{x}_0\}\cup\{\hat{x}_0\}$, then the arc $\check{a}$ is the incoming arc at $\hat{c}_i$, and one of the arc among $\check{a}$ and $\hat{a}$ is the over arc at $\hat{c}_i$. In other words, in the diagram $D_2$, the arc $a$ is the incoming as well as over arc at $\hat{c}_i$. This is again a contradiction, and hence $\hat{u}_i\neq\hat{v}_i$. Note that $\hat{y}_i=\hat{u}_i*\hat{v}_i$. By Proposition \ref{prop10}, we have $\hat{u}_i\, \blackdiamond_i\, \hat{y}_i\, \blackdiamond_i\, \hat{v}_i$ for some $\blackdiamond_i\in\{<,>\}$. This implies that $\hat{z}_1<\hat{y}_1$ for some $\hat{z}_1\in\{\hat{u}_1,\hat{v}_1\}$ and $\hat{y}_2<\hat{z}_2$ for some $\hat{z}_2\in\{\hat{u}_2,\hat{v}_2\}$. In other words $\hat{z}_1<\hat{y}_1<\hat{y}_2<\hat{z}_2$ for some $\hat{z}_1,\hat{z}_2\in\{\hat{x}_0,x_1,x_2,\ldots,x_n\}\cup\{\check{x}_0\}$. But, then at least one of the elements $\hat{z}_1$ or $\hat{z}_2$ must belong to $\{\hat{x}_0,x_1,x_2,\ldots,x_n\}$. This contradicts the choice of at least one of $\hat{y}_1$ or $\hat{y}_2$.

\begin{figure}[H]
\centering
\begin{minipage}{0.47\textwidth}
\centering
\begin{tikzpicture}[scale=0.6]
\node at (0.4,-1.2) {{\small $\hat{u}_i$}};
\node at (-1.2,0.4) {{\small $\hat{v}_i$}};
\node at (0.4,1.2) {{\small $\hat{y}_i$}};
\begin{knot}[clip width=6, clip radius=4pt]
\strand[->] (-2,0)--(2,0);
\strand[->] (0,-2)--(0,2);
\end{knot}
\end{tikzpicture}
\caption{At the crossing $\hat{c}_i$}
\label{fig7}
\end{minipage}
\begin{minipage}{0.47\textwidth}
\centering
\begin{tikzpicture}[scale=0.6]
\node at (0.4,1.2) {{\small $\check{u}_i$}};
\node at (-1.2,0.4) {{\small $\check{v}_i$}};
\node at (0.4,-1.2) {{\small $\check{y}_i$}};
\begin{knot}[clip width=6, clip radius=4pt]
\strand[->] (2,0)--(-2,0);
\strand[->] (0,-2)--(0,2);
\end{knot}
\end{tikzpicture}
\caption{At the crossing $\check{c}_i$}
\label{fig8}
\end{minipage}
\end{figure}

\item $L_2$ is a negative link: Let $\check{y}_1$ and $\check{y}_2$ be the smallest and largest elements, respectively, in $\{\check{x}_0,x_1,x_2,\ldots,x_n\}$. For $i=1,2$, consider the crossing $\check{c}_i$ where the arc corresponding to $\check{y}_i$ is an incoming arc (see Figure \ref{fig8}). Note that $\check{c}_i$ must be a crossing in $D_2$. Let $\check{u}_i\in\{\hat{x}_0,x_1,x_2,\ldots,x_n\}$ and $\check{v}_i\in\{\check{x}_0,x_1,x_2,\ldots,x_n\}\cup\{\hat{x}_0\}$ be the elements corresponding to the outgoing arc and the over arc at $\check{c}_i$, respectively (see Figure \ref{fig8}). By the similar argument as in the first case, we have $\check{y}_i\neq\check{v}_i$. Note that $\check{u}_i=\check{y}_i*^{-1}\check{v}_i$. By Proposition \ref{prop10}, we have $\check{u}_i\, \blackdiamond_i\, \check{y}_i\, \blackdiamond_i\, \check{v}_i$ for some $\blackdiamond_i\in\{<,>\}$. Now, arguing as in the first case leads to a contradiction.\qedhere
\end{enumerate}
\end{proof}

For rational numbers $r_1,r_2,\ldots,r_k$, the {\it Montesinos link} $M(r_1,r_2,\ldots,r_k)$ is the link shown in Figure \ref{fig9}, where $t(r_i)$ is the rational tangle \cite{aps,kl} associated with $r_i$ for $i=1,2,\ldots,k$. If $n_1,n_2,\ldots,n_k$ are integers, then the Montesinos link $M(1/n_1,1/n_2,\ldots,1/n_k)$ is called the {\it pretzel link} of type $(n_1,n_2,\ldots,n_k)$. Note that any $2$-bridge link (i.e. a rational link) is a Montesinos link.

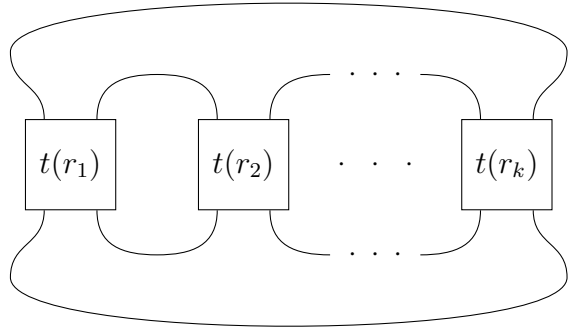
\begin{figure}[H]
\centering
\begin{tikzpicture}
\node (TR1) at (-2.9,0) [draw, minimum width=12mm, minimum height=12mm]{{\large $t(r_1)$}};
\node (TR2) at (-0.6,0) [draw, minimum width=12mm, minimum height=12mm]{{\large $t(r_2)$}};
\node (TRk) at (2.9,0) [draw, minimum width=12mm, minimum height=12mm]{{\large $t(r_k)$}};
\begin{knot}[clip width=6, clip radius=4pt]
\strand[-] (TR1.120) to [out=up, in=down, looseness=1.2] (-3.7,1.5) to [out=up, in=up, looseness=0.3] (3.7,1.5) to [out=down, in=up, looseness=1.2] (TRk.60);
\strand[-] (TR1.240) to [out=down, in=up, looseness=1.2] (-3.7,-1.5) to [out=down, in=down, looseness=0.3] (3.7,-1.5) to [out=up, in=down, looseness=1.2] (TRk.300);
\strand[-] (TR1.60) to [out=up, in=left, looseness=1.2] (-1.75,1.2) to [out=right, in=up, looseness=1.2] (TR2.120);
\strand[-] (TR1.300) to [out=down, in=left, looseness=1.2] (-1.75,-1.2) to [out=right, in=down, looseness=1.2] (TR2.240);
\strand[-] (TR2.60) to [out=up, in=left, looseness=1.2] (0.55,1.2);
\strand[-] (1.75,1.2) to [out=right, in=up, looseness=1.2] (TRk.120);
\strand[-] (TR2.300) to [out=down, in=left, looseness=1.2] (0.55,-1.2);
\strand[-] (1.75,-1.2) to [out=right, in=down, looseness=1.2] (TRk.240);
\end{knot};
\node at (0.7,0) {$\cdot$}; \node at (1.15,0) {$\cdot$}; \node at (1.6,0) {$\cdot$};
\node at (0.85,1.2) {$\cdot$}; \node at (1.15,1.2) {$\cdot$}; \node at (1.45,1.2) {$\cdot$};
\node at (0.85,-1.2) {$\cdot$}; \node at (1.15,-1.2) {$\cdot$}; \node at (1.45,-1.2) {$\cdot$};
\end{tikzpicture}
\caption{Montesinos link $M(r_1,r_2,\ldots,r_k)$}
\label{fig9}
\end{figure}

\begin{corollary}\label{cor5}
Let $M$ be a non-trivial Montesinos link that is prime, alternating and positive (or negative). Then the link quandle of a connected sum of $M$ with any link is not bi-orderable. In particular, the link quandle of $M$ is not bi-orderable.
\end{corollary}

\begin{proof}
Consider an alternating diagram $D$ of $M$ without a nugatory crossing (i.e. $D$ is a minimal diagram of $M$). By \cite[Corollary 2]{n2}, the diagram $D$ is positive, and hence it is a minimal positive diagram of $M$. Let $x_0,x_1,\ldots,x_n$ be generators of the link quandle $Q(M)$ corresponding to arcs in $D$. Suppose $H_1(X_{M},\mathbb{Z})$ be the first homology group of the double branched cover $X_{M}$ of $\mathbb{S}^3$ branched along $M$. Then, by \cite[Theorem 4.2]{aps}, different arcs of $D$ represent different elements of $H_1(X_{M},\mathbb{Z})$. This is equivalent to the statement that for any pair of arcs of the diagram $D$, there is a coloring by elements of $\Core(H_1(X_{M},\mathbb{Z}))$ distinguishing them. Hence, the elements $x_0,x_1,\ldots,x_n$ in $Q(M)$ are all distinct. Taking $M$ in place of $L_2$ and $D$ in place of $D_2$, the result now follows from Theorem \ref{th6}.
\end{proof}

As examples, knot quandles of knots $3_1$, $5_1$ and $5_2$ (and of their mirror images) are not bi-orderable, since each of them is a positive (or a negative) alternating rational knot.

\begin{corollary}\label{cor6}
Let $K$ be an alternating and positive (or negative) knot of prime determinant. Then the link quandle of a connected sum of $K$ with any link is not bi-orderable. In particular, the knot quandle of $K$ is not bi-orderable.
\end{corollary}

\begin{proof}
Consider a minimal diagram $D$ of $K$. By \cite[Corollary 2]{n2}, the diagram $D$ is positive, and hence it is a minimal positive diagram of $K$. Let $x_0,x_1,\ldots,x_n$ be generators of the knot quandle $Q(K)$ corresponding to arcs in $D$. Then, by \cite[Proposition 4.4]{ms}, there exists a Fox coloring that assigns different colors to different arcs of the diagram $D$. Thus, the elements $x_0,x_1,\ldots,x_n$ in $Q(K)$ are also distinct. The result now follows from Theorem \ref{th6}.
\end{proof}

\medskip

\section{Orderability of link quandles of torus links}\label{orderability-torus-quandle}

Recall that two links $L_1$ and $L_2$ are called {\it weakly equivalent} if $L_1$ is ambient isotopic to either $L_2$ or the reverse of the mirror image of $L_2$. It is known that link quandles of weakly equivalent links are isomorphic (see \cite[Theorem 5.2 and Corollary 5.3]{frk}). For any $m,n\geq1$, since the torus link $T(m,n)$ is invertible, it is weakly equivalent to its reverse, mirror image and the reverse of its mirror image, and hence the link quandles of all of them are isomorphic to that of $T(m,n)$. Recall that a torus link $T(m,n)$ is a knot (a one component link) if and only if $\gcd(m,n)=1$.

\begin{proposition}\label{prop5}
The link quandle of a torus link $T(m,n)$ is generated by $a_1,a_2,\ldots,a_m$ and has the following relations:
\begin{equation*}
a_i=a_{n+i}*a_n*a_{n-1}*\cdots*a_1\quad\textrm{for}\;\;i=1,2,\ldots,m,
\end{equation*}
where $a_{mj+k}=a_k$ for $j\in\mathbb{Z}$ and $k\in\{1,2,\ldots,m\}$.
\end{proposition}

\begin{proof}
Since a torus link $T(m,n)$ is the closure of the braid $\tau(m,n)=\left(\sigma_1\sigma_2\cdots\sigma_{m-1}\right)^n$, with reference to Figure \ref{fig2}, it is enough to prove that
\begin{equation}\label{eq5}
c_i=a_{n+i}*a_n*a_{n-1}*\cdots*a_1\quad\textrm{for}\;\;i=1,2,\ldots,m.
\end{equation}

\begin{figure}[H]
\centering
\begin{minipage}{0.495\textwidth}
\centering
\begin{tikzpicture}[scale=0.5]
\node (a1) at (-3.5,-7.6) {$a_1$};
\node (a2) at (-2,-7.6) {$a_2$};
\node (a3) at (-0.5,-7.6) {$a_3$};
\node at (0.7,-7.6) {\large{$\cdot$}};
\node at (1.4,-7.6) {\large{$\cdot$}};
\node at (2.1,-7.6) {\large{$\cdot$}};
\node (am) at (3.5,-7.6) {$a_m$};
\node[rect1] at (0,-4.9) {$\tau(m,n)$};
\node (c1) at (-3.5,-2) {$c_1$};
\node (c2) at (-2,-2) {$c_2$};
\node (c3) at (-0.5,-2) {$c_3$};
\node at (0.7,-2) {\large{$\cdot$}};
\node at (1.4,-2) {\large{$\cdot$}};
\node at (2.1,-2) {\large{$\cdot$}};
\node (cm) at (3.5,-2) {$c_m$};
\draw (a1)--(-3.5,-6.6) (a2)--(-2,-6.6) (a3)--(-0.5,-6.6) (am)--(3.5,-6.6);
\draw[arrow, rounded corners=7pt] (-3.5,-3.2)--(c1);
\draw[arrow, rounded corners=7pt] (-2,-3.2)--(c2);
\draw[arrow, rounded corners=7pt] (-0.5,-3.2)--(c3);
\draw[arrow, rounded corners=7pt] (3.5,-3.2)--(cm);
\end{tikzpicture}
\caption{Toric braid $\tau(m,n)$}
\label{fig2}
\end{minipage}
\begin{minipage}{0.495\textwidth}
\centering
\begin{tikzpicture}[scale=0.7]
\node (a1) at (-3.5,-2) {$a_1$};
\node (a2) at (-2,-2) {$a_2$};
\node (a3) at (-0.5,-2) {$a_3$};
\node at (0.7,-2) {\large{$\cdot$}};
\node at (1.4,-2) {\large{$\cdot$}};
\node at (2.1,-2) {\large{$\cdot$}};
\node (am) at (3.5,-2) {$a_m$};
\node (c1) at (-3.5,2) {$c_1$};
\node (c2) at (-2,2) {$c_2$};
\node at (-0.8,2) {\large{$\cdot$}};
\node at (-0.1,2) {\large{$\cdot$}};
\node at (0.6,2) {\large{$\cdot$}};
\node (cmo) at (2,2) {$c_{m-1}$};
\node (cm) at (3.5,2) {$c_m$};
\draw[arrow, rounded corners=7pt] (a2)--(-2,-1.05)--(-3.5,-0.6)--(c1);
\draw[arrow, rounded corners=7pt] (a3)--(-0.5,-0.6)--(-2,-0.15)--(c2);
\draw[arrow, rounded corners=7pt] (am)--(3.5,0.6)--(2,1.05)--(cmo);
\draw[arrow, rounded corners=7pt, preaction={draw=white,line width=5pt}] (a1)--(-3.5,-1.05)--(3.5,1.05)--(cm);
\end{tikzpicture}
\caption{Toric braid $\tau(m,1)$}
\label{fig3}
\end{minipage}
\end{figure}

\noindent We prove \eqref{eq5} by induction on $n$. By looking at Figure \ref{fig3}, one can see that $c_i=a_{i+1}*a_1$ for $i=1,2,\ldots, m$. Thus, the equations given by \eqref{eq5} hold for $n=1$. Assume that the equations given by \eqref{eq5} hold for a positive integer $n-1$. Since $\tau(m,n)=\tau(m,n-1)\tau(m,1)$ (see Figure \ref{fig4}), we have 
\begin{equation}\label{eq6}
c_i=b_{i+1}*b_1 \quad\textrm{for}\;\;i=1,2,\ldots,m
\end{equation}
where $b_{m+1}=b_1$. By induction hypothesis,
\begin{equation}\label{eq7}
b_{i+1}=a_{n+i}*a_{n-1}*a_{n-2}*\cdots*a_1\quad\textrm{for}\;\;i=1,2,\ldots,m.
\end{equation}

\begin{figure}[H]
\centering
\begin{tikzpicture}[scale=0.7]
\node (a1) at (-3.5,-7.6) {$a_1$};
\node (a2) at (-2,-7.6) {$a_2$};
\node (a3) at (-0.5,-7.6) {$a_3$};
\node at (0.7,-7.6) {\large{$\cdot$}};
\node at (1.4,-7.6) {\large{$\cdot$}};
\node at (2.1,-7.6) {\large{$\cdot$}};
\node (am) at (3.5,-7.6) {$a_m$};
\node[rect2] at (0,-4.8) {$\tau(m,n-1)$};
\node (b1) at (-3.5,-2) {$b_1$};
\node (b2) at (-2,-2) {$b_2$};
\node (b3) at (-0.5,-2) {$b_3$};
\node at (0.7,-2) {\large{$\cdot$}};
\node at (1.4,-2) {\large{$\cdot$}};
\node at (2.1,-2) {\large{$\cdot$}};
\node (bm) at (3.5,-2) {$b_m$};
\node (c1) at (-3.5,2) {$c_1$};
\node (c2) at (-2,2) {$c_2$};
\node at (-0.8,2) {\large{$\cdot$}};
\node at (-0.1,2) {\large{$\cdot$}};
\node at (0.6,2) {\large{$\cdot$}};
\node (cmo) at (2,2) {$c_{m-1}$};
\node (cm) at (3.5,2) {$c_m$};
\draw (a1)--(-3.5,-6.8) (-3.5,-2.8)--(b1) (a2)--(-2,-6.8) (-2,-2.8)--(b2) (a3)--(-0.5,-6.8) (-0.5,-2.8)--(b3) (am)--(3.5,-6.8) (3.5,-2.8)--(bm);
\draw[arrow, rounded corners=7pt] (b2)--(-2,-1.05)--(-3.5,-0.6)--(c1);
\draw[arrow, rounded corners=7pt] (b3)--(-0.5,-0.6)--(-2,-0.15)--(c2);
\draw[arrow, rounded corners=7pt] (bm)--(3.5,0.6)--(2,1.05)--(cmo);
\draw[arrow, rounded corners=7pt, preaction={draw=white,line width=5pt}] (b1)--(-3.5,-1.05)--(3.5,1.05)--(cm);
\end{tikzpicture}
\caption{Toric braid $\tau(m,n)$} seen as $\tau(m,n-1)\tau(m,1)$
\label{fig4}
\end{figure}

\noindent Using \eqref{eq7} in \eqref{eq6}, we get
\begin{align}
c_i&=\left(a_{n+i}*a_{n-1}*a_{n-2}*\cdots*a_1\right)*\left(a_n*a_{n-1}*a_{n-2}*\cdots*a_1\right)\\
&=a_{n+i}*a_{n-1}*a_{n-2}*\cdots*a_1*^{-1}a_1*^{-1}a_2*^{-1}\cdots*^{-1}a_{n-1}*a_n*a_{n-1}*\cdots*a_1\\
&=a_{n+i}*a_n*a_{n-1}*\cdots*a_1\quad\;\textrm{for}\;\;i=1,2,\ldots,m
\end{align}
where the second equality follows from Lemma \ref{lem4} and the third follows by the cancellation. This proves that the equations given by \eqref{eq5} hold for all $n$.
\end{proof}

If $<$ is a right ordering on a quandle $Q$ and $x,y,z_1,z_2,\ldots,z_n\in Q$ with $x\,\blackdiamond\, y$ for $\blackdiamond\in\{<,>\}$, then
\begin{equation}\label{eq14}
x*z_1*z_2*\cdots*z_n\,\blackdiamond\, y*z_1*z_2*\cdots*z_n~\textrm{and}~
x*^{-1}z_1*^{-1}z_2*^{-1}\cdots*^{-1}z_n\,\blackdiamond\, y*^{-1}z_1*^{-1}z_2*^{-1}\cdots*^{-1}z_n.
\end{equation}

\begin{theorem}\label{th4}
Let $m,n\geq2$ be integers such that one is not a multiple of the other. Then the link quandle of the torus link $T(m,n)$ is not right-orderable.
\end{theorem}

\begin{proof}
Note that the torus links $T(m,n)$ and $T(n,m)$ are ambient isotopic. Thus, we can assume that $m<n$ by switching $m$ and $n$ if required. Let $d=\gcd(m,n)$. Then $d<m$. By Proposition \ref{prop5}, the link quandle $Q(T(m,n))$ is generated by $a_1,a_2,\ldots,a_m$ and has the relations
\begin{equation}\label{eq10}
a_i=a_{n+i}*a_n*a_{n-1}*\cdots*a_1\quad\textrm{for}\;\;i=1,2,\ldots,m,
\end{equation}
where $a_{mj+k}=a_k$ for $j\in\mathbb{Z}$ and $k\in\{1,2,\ldots,m\}$. Using \eqref{eq10}, one obtains
\begin{equation}\label{eq11}
a_i=a_{n+i}*a_n*a_{n-1}*\cdots*a_1\quad\textrm{for all}\;\;i\in\mathbb{Z},
\end{equation}
where $a_{mj+k}=a_k$ for $j\in\mathbb{Z}$ and $k\in\{1,2,\ldots,m\}$. We can rewrite \eqref{eq11} as
\begin{equation}\label{eq12}
a_{i-n}=a_i*a_n*a_{n-1}*\cdots*a_1\quad\textrm{for all}\;\;i\in\mathbb{Z}.
\end{equation}
\noindent Also, \eqref{eq11} can be written as 
\begin{equation}\label{eq13}
a_{n+i}=a_i*^{-1}a_1*^{-1}a_2*^{-1}\cdots*^{-1}a_n\quad\textrm{for all}\;\;i\in\mathbb{Z}.
\end{equation}
	
Suppose on the contrary that the quandle $Q(T(m,n))$ is right-ordered with respect to a linear order $<$. By the proof of Proposition \ref{prop5} (see figures \ref{fig2}, \ref{fig3} and \ref{fig4}), the generators $a_1,a_2,\ldots,a_m$ of $Q(T(m,n))$ correspond to some of the arcs in the standard diagram of the closed toric braid representing $T(m,n)$. Note that $\eta(a_1),\eta(a_2),\ldots,\eta(a_m)$ are the meridional elements that generate the link group $G(T(m,n))$, where $\eta:Q(T(m,n))\to G(T(m,n))$ is the natural map. According to \cite[Corollary 1.5]{rz}, the elements $\eta(a_1),\eta(a_2),\ldots,\eta(a_m)$ must be pairwise distinct in $G(T(m,n))$, and hence so are the elements $a_1,a_2,\ldots,a_m$ in $Q(T(m,n))$. In particular, we have $a_1\neq a_{d+1}$, and hence $a_1\,\blackdiamond\, a_{d+1}$ for some $\blackdiamond\in\{<,>\}$. A repeated application of \eqref{eq14} together with \eqref{eq12} and \eqref{eq13} yields
\begin{equation}\label{eq15}
a_{nk+1}\,\blackdiamond\, a_{nk+d+1}\quad\textrm{for all}\;\;k\in\mathbb{Z}.
\end{equation}
Let $l$ be an integer. Since $\gcd(m,n)=d$, we have $dl=mj+nk$ for some integers $j$ and $k$. This implies that $nk+1\equiv dl+1\;(\!\!\!\!\mod m)$ and $nk+d+1\equiv dl+d+1\;(\!\!\!\!\mod m)$. By \eqref{eq15}, we have $a_{dl+1}\,\blackdiamond\, a_{dl+d+1}$. Thus, $a_{dl+1}\,\blackdiamond\, a_{dl+d+1}$ for any integer $l$. Using this repeatedly, we get $a_1\,\blackdiamond\, a_{d+1}\,\blackdiamond\, a_{2d+1}\,\blackdiamond\,\cdots\,\blackdiamond\, a_{cd+1}\,\blackdiamond\, a_1$, where $c=\frac{m}{d}-1$. This implies that $a_1<a_1$ or $a_1>a_1$, a contradiction.
\end{proof}

As a consequence of the preceding theorem, we retrieve the following result of Perron and Rolfsen \cite[Proposition 3.2]{pr}.

\begin{corollary}\label{cor3}
The knot group of a non-trivial torus knot is not bi-orderable.
\end{corollary}

\begin{proof}
Let $K$ be a non-trivial torus knot. Then, by Theorem \ref{th4}, the knot quandle of $K$ is not right-orderable, and hence by Corollary \ref{cor1}, the knot group of $K$ is not bi-orderable.
\end{proof}

We conclude with the following result.

\begin{corollary}
The knot quandle of the trefoil knot is neither left nor right-orderable.
\end{corollary}

\begin{proof}
Note that the trefoil knot is the torus knot $T(2,3)$. By Theorem \ref{th4}, the knot quandle of $T(2,3)$ is not right-orderable. We claim that the knot quandle of $T(2,3)$ is not left-orderable as well. Using Proposition \ref{prop5}, the knot quandle $Q\left(T(2,3)\right)$ is generated by $a_1$ and $a_2$ with defining relations $a_1=a_2*a_1*a_2*a_1$ and $a_2=a_1*a_1*a_2*a_1$. These relations can be rewritten as
\begin{align}
a_1&=a_2*a_1*a_2\quad\;\textrm{and}\label{eq8}\\
a_2&=a_1*a_2*a_1.\label{eq9}
\end{align}
Assume contrary that the quandle $Q\left(T(2,3)\right)$ is left-ordered with respect to a linear order $<$. Since $Q\left(T(2,3)\right)$ is non-trivial, we must have $a_1\neq a_2$, and hence $a_1\,\blackdiamond\, a_2$ for some $\blackdiamond\in\{<,>\}$. Consider
\begin{align}
&\quad a_1\,\blackdiamond\, a_2&&\\
\Rightarrow&\quad a_1*a_1\,\blackdiamond\, a_1*a_2&&\textrm{(since $<$ is left ordering)}\\
\Rightarrow&\quad a_1\,\blackdiamond\, a_1*a_2&&\textrm{(by idempotency)}\\
\Rightarrow&\quad a_2*a_1\,\blackdiamond\, a_2*(a_1*a_2)&&\textrm{(since $<$ is left ordering)}\\
\Rightarrow&\quad a_2*a_1\,\blackdiamond\, a_2*a_1*a_2&&\textrm{(by Lemma \ref{lem4})}\\
\Rightarrow&\quad a_2*a_1\,\blackdiamond\, a_1&&\textrm{(by \eqref{eq8})}\\
\Rightarrow&\quad a_1*(a_2*a_1)\,\blackdiamond\, a_1*a_1&&\textrm{(since $<$ is left ordering)}\\
\Rightarrow&\quad a_1*a_2*a_1\,\blackdiamond\, a_1&&\textrm{(by Lemma \ref{lem4})}\\
\Rightarrow&\quad a_2\,\blackdiamond\, a_1&&\textrm{(by \eqref{eq9})}.
\end{align}
This is a contradiction as we cannot have $a_1\,\blackdiamond\, a_2$ and $a_2\,\blackdiamond\, a_1$ together.
\end{proof}

\medskip

\section{Orderability of involutory quandles of alternating links}\label{involutory-quandles-links}
We know that a non-trivial involutory quandle is not right-orderable, whereas there are many involutory quandles that are left-orderable. For example, the quandle $\Core(G)$ is left-orderable for any bi-orderable group $G$. We conclude with some observations on left-orderability of involutory quandles of alternating links.

\begin{theorem}\label{thm-left-order}
Let $L$ be a non-trivial alternating link. If there exists a reduced alternating diagram $D$ of $L$ such that generators of the involutory quandle $IQ(L)$ of $L$ corresponding to arcs in $D$ are pairwise distinct, then $IQ(L)$ is not left-orderable.
\end{theorem}

\begin{proof}
First suppose that $L$ is a non-trivial non-split alternating link. Let $D$ be a reduced alternating diagram of $L$ such that generators of the involutory quandle $IQ(L)$ of $L$ corresponding to arcs in $D$ are pairwise distinct.	Suppose on the contrary that the involutory quandle $IQ(L)$ is left-ordered with respect to a linear order $<$. Let $y$ be the smallest element among the generators of $IQ(L)$ corresponding to arcs in $D$. Since $L$ is non-trivial and non-split, and $D$ is alternating, there is a crossing, say $c$, in $D$ such that the arc corresponding to $y$ is the over arc at $c$ (see Figure \ref{fig10}). Let $x$ and $z$ be the elements of $IQ(L)$ corresponding to the other two arcs meeting at $c$ (see Figure \ref{fig10}).
\begin{figure}[H]
\centering
\begin{tikzpicture}[scale=0.6]
\node at (-1.2,0.4) {{\small$x$}};
\node at (0.4,-1.2) {{\small$y$}};
\node at (1.2,0.4) {{\small$z$}};
\begin{knot}[clip width=6, clip radius=4pt]
\strand[-] (0,-1.8)--(0,1.8);
\strand[-] (-2,0)--(2,0);
\end{knot}
\end{tikzpicture}
\caption{At the crossing $c$}
\label{fig10}
\end{figure}
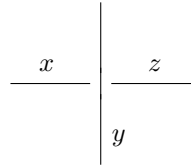
By the hypothesis, we have $x\neq y$ and $y\neq z$. Now, $y$ being the smallest element implies that $y< x $ and $y <z$. Since $<$ is a left-ordering on $IQ(L)$ this implies that $x*y < x*x=x$ and $z*y < z*z=z$. Since the quandle is involutory, we have $x*y=z$ and $z*y=x$, and hence $z<x$ and $x<z$, a contradiction.
\par

Now suppose that $L$ is an arbitrary non-trivial alternating link. Then there exists a non-trivial non-split component, say $L'$, of $L$. Let $D'$ be the diagram of $L'$ obtained from $D$ by throwing away the components of $L$ that do not belong to $L'$. Note that the involutory quandle $IQ(L')$ of $L'$ is a subquandle of $IQ(L)$. By the preceding paragraph, it follows that $IQ(L')$ is not left-orderable, and hence so is $IQ(L)$.
\end{proof}

\begin{corollary}
Let $M$ be a non-trivial Montesinos link that is prime and alternating. Then the involutory quandle $IQ(M)$ of $M$ is not left-orderable.
\end{corollary}

\begin{proof}
Consider a reduced alternating diagram $D$ of $M$. By the same argument as in the proof of Corollary \ref{cor5}, the elements of $IQ(M)$ corresponding to the arcs in $D$ are pairwise distinct. The result now follows from Theorem \ref{thm-left-order}.
\end{proof}

\begin{corollary}
Let $K$ be an alternating knot of prime determinant. Then the involutory quandle $IQ(K)$ of $K$ is not left-orderable.
\end{corollary}

\begin{proof}
Consider a reduced alternating diagram $D$ of $K$. Arguing as in the proof of Corollary \ref{cor6}, the elements of $IQ(M)$ corresponding to any two arcs in $D$ are distinct. The proof now follows from Theorem \ref{thm-left-order}.
\end{proof}

Let $m\geq3$ and $n\geq2$ be relatively prime integers. The {\it Turk's head knot} $THK(m,n)$ is the closure of the braid $\left(\sigma_1\sigma_2^{-1}\sigma_3\sigma_4^{-1}\cdots\sigma_{m-1}^\delta\right)^n$, where $\delta=+1$ if $m$ is even and $\delta=-1$ if $m$ is odd. Note that Turk's head knots are alternating knots.

\begin{corollary}
Let $m\geq3$ and $n\geq2$ be relatively prime integers. If $n=2$ or $m=3$ or $m$ is even, then the involutory quandle of the Turk's head knot $THK(m,n)$ is not left-orderable.
\end{corollary}

\begin{proof}
Let $D$ be a reduced alternating diagram of $THK(m,n)$. By \cite[Theorems 3 and 4]{dmms}, there exists a Fox coloring that assigns different colors to different arcs of the diagram $D$. Hence, elements of the involutory quandle of $THK(m,n)$ corresponding to the arcs in $D$ are pairwise distinct. The result now follows from Theorem \ref{thm-left-order}.
\end{proof}

\medskip

\begin{ack}
Hitesh Raundal is supported by research associateship under the SERB grant SB/ SJF/2019-20. Manpreet Singh thanks IISER Mohali for the PhD Research Fellowship. Mahender Singh is supported by the Swarna Jayanti Fellowship grants DST/SJF/MSA-02/2018-19 and SB/SJF/2019-20.
\end{ack}

\end{document}